\documentclass[11pt,twoside]{amsart}

\usepackage{enumerate,color,xypic} 
\usepackage{setspace}
\usepackage[dvips,final]{graphics}

\usepackage[hyperindex=true,plainpages=false,colorlinks=false,pdfpagelabels]{hyperref}

\theoremstyle{plain}
\newtheorem*{theorem*}{Theorem}
\newtheorem{theorem}{Theorem}[section]
\newtheorem{lemma}[theorem]{Lemma}
\newtheorem{cor}[theorem]{Corollary}
\newtheorem{prop}[theorem]{Proposition}

\theoremstyle{definition}
\newtheorem{definition}[theorem]{Definition}

\theoremstyle{remark}
\newtheorem{rem}[theorem]{Remark}

\numberwithin{equation}{section}

\renewcommand{\Im}{{\rm Im}\,}

\newcommand{\R}{\mathbb{ R}}
\newcommand{\C}{\mathbb{ C}}
\newcommand{\Z}{\mathbb{ Z}}

\newcommand{\F}{\mathbb{K}}

\renewcommand{\H}{\mathbb{ H}}

\newcommand{\N}{\mathbb{ N}}
\renewcommand{\P}{\mathbb{ P}}
\newcommand{\HP}{\H\P}
\newcommand{\CP}{\C\P}

\newcommand{\dbar}{{\bar{\partial}}}

\DeclareMathOperator{\End}{End}
\DeclareMathOperator{\Hom}{Hom}

\DeclareMathOperator{\Span}{Span}

\DeclareMathOperator{\ev}{ev}

\DeclareMathOperator{\Spec}{Spec}

\DeclareMathOperator{\Ind}{Index}

\renewcommand{\i}{\mathbf{i}} 

\begin{document}

\title[Discrete holomorphic geometry I]{Discrete holomorphic geometry I.\\
  Darboux transformations and spectral curves} \date{\today} \author{C.~Bohle}
\author{F.~Pedit} \author{U.~Pinkall}

\address{Christoph Bohle\\
  Institut f\"ur Mathematik\\
  Technische Universit\"at Berlin\\
  Strasse des 17.Juni 136\\
  10623 Berlin, Germany}

\address{Franz Pedit\\
  Mathematisches Institut der  Universit{\"a}t T\"ubingen\\
  Auf der Morgenstelle 10\\
  72076 T\"ubingen\\
  Germany\\
  and \\
  Department of Mathematics \\
  University of Massachusetts\\
  Amherst, MA 01003, USA}

\address{Ulrich Pinkall\\
  Institut f\"ur Mathematik\\
  Technische Universit\"at Berlin\\
  Strasse des 17.Juni 136\\
  10623 Berlin, Germany}

\email{bohle@math.tu-berlin.de,
  pedit@mathematik.uni-tuebingen.de,\newline  
  pinkall@math.tu-berlin.de}

\thanks{First and third author supported by DFG Forschergruppe 565
  ``Polyhedral Surfaces''. Second and third author supported by DFG Research
  Center ``Matheon''. All authors additionally supported by DFG
  Schwerpunktprogramm 1154 ``Global Differential Geometry''.}

\begin{abstract}
  Finding appropriate notions of discrete holomorphic maps and, more
  generally, conformal immersions of discrete Riemann surfaces into $3$--space
  is an important problem of discrete differential geometry and computer
  visualization.  We propose an approach to discrete conformality that is
  based on the concept of holomorphic line bundles over ``discrete surfaces'',
  by which we mean the vertex sets of triangulated surfaces with bi--colored
  set of faces.  The resulting theory of discrete conformality is
  simultaneously M\"obius invariant and based on linear equations.  In the
  special case of maps into the $2$--sphere we obtain a reinterpretation of
  the theory of complex holomorphic functions on discrete surfaces introduced
  by Dynnikov and Novikov.

  As an application of our theory we introduce a Darboux transformation for
  discrete surfaces in the conformal $4$--sphere. This Darboux transformation
  can be interpreted as the space-- and time--discrete Davey--Stewartson flow
  of Konopelchenko and Schief. For a generic map of a discrete torus with
  regular combinatorics, the space of all Darboux transforms has the structure
  of a compact Riemann surface, the spectral curve.  This makes contact to the
  theory of algebraically completely integrable systems and is the starting
  point of a soliton theory for triangulated tori in $3$-- and $4$--space
  devoid of special assumptions on the geometry of the surface.
  \end{abstract}

\maketitle

\section{Introduction}\label{sec:introduction}

The notions of discrete Riemann surfaces and discrete conformal maps are
important recurring themes in discrete geometry.  In computer graphics,
discrete conformal parameterizations and their approximations are used as
texture mappings and for the construction of geometric images.  In
mathematical physics, discrete Riemann surfaces occur in the discretization of
physical models such as conformal field theories and statistical mechanics
which involve smooth Riemann surfaces.  In surface geometry, the concept of
discrete conformality is fundamental in the description of discrete analogues
of special surface classes including minimal and, more generally, constant
mean or Gaussian curvature surfaces.

Whereas discrete models generally have smooth limits there are no consistent
procedures to ``discretize'' smooth models (which is reminiscent to the
problem of ``quantizing'' a classical theory).  There are currently several
approaches to discrete conformality for maps into the complex plane including
the M\"obius invariant approach modeled on circle packings or patterns which
goes back to Koebe and more recently Thurston, cf.~\cite{St}, M\"obius
invariant polygon evolutions driven by constant cross--ratio conditions
\cite{BP96,BP99,HMNP01} and linear approaches modelled on discretizations of
the Cauchy--Riemann equation, see e.g.\ Dynnikov and Novikov~\cite{DN97,DN03},
Mercat~\cite{Me}, Kenyon~\cite{Ke02}. For a comparison to the circle packing
approach and integrable system interpretation, see Bobenko, Mercat, and
Suris~\cite{BMS05}. Discrete conformal maps into space are often modeled on
discrete versions of conformal curvature lines and thus apply only to the
restricted class of isothermic surfaces, cf.~\cite{BP96,BP99,BHS06}.  The
conformality condition is again expressed by cross--ratio conditions, that is
by non--linear difference equations on the vertex positions.

The theory of discrete holomorphicity developed in the present paper provides
an approach to conformal geometry of discrete surfaces which applies equally
to maps into the plane and higher dimensional target spaces and is M\"obius
invariant, given by linear equations, and not restricted to special surface
classes. Our approach is based on the concept of holomorphic line bundles over
discrete surfaces.  The discrete conformal maps are then ratios of holomorphic
sections of such line bundles. From this point of view, the only difference
between planar conformal maps, i.e., holomorphic functions, and conformal maps
into $3$-- and $4$--space is whether the holomorphic line bundles in question
are complex or quaternionic.  The relation between the M\"obius geometry of
discrete surfaces and discrete holomorphic line bundles is provided by a
discrete version of the Kodaira correspondence.  A vantage point of this
approach to discrete holomorphic geometry is that the definition of
holomorphic line bundles does not require an a priori notion of discrete
Riemann surfaces about which we have nothing to say. This is indicated by the
smooth theory where any linear, first--order elliptic differential operator
between complex or quaternionic line bundles over an oriented $2$--dimensional
manifold $M$ defines a complex structure on $M$ and renders the line bundles
holomorphic.

The notion of holomorphic line bundles over discrete surfaces requires an
additional combinatorial structure which, in the context of discrete
conformality, was introduced by Dynnikov and Novikov \cite{DN97,DN03} and
appears also in Colin de Verdi\`ere~\cite{CdV98} and previously in unpublished
notes by Thurston: in the present paper a {\em discrete surface} $M$ is the
vertex set of a triangulated smooth surface with a bi--coloring of the faces
into black and white triangles, see Figure~\ref{fig:cliff}.
\begin{figure}[hbt]
  \centering
  \resizebox{6cm}{4.5cm}{\includegraphics{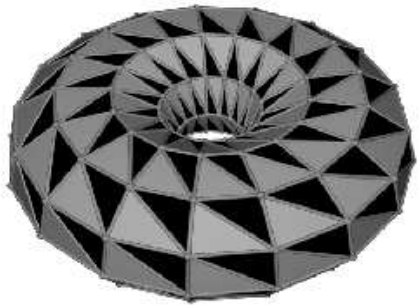}}
  \caption{A discrete torus immersed into 3--space.}
  \label{fig:cliff}
\end{figure}
  A complex or quaternionic {\em line bundle} $W$ over
such a discrete surface $M$ is then a family of $1$--dimensional complex or
quaternionic vector spaces $W_p$ parameterized by the vertices $p\in M$. A
holomorphic structure on a line bundle $W$ over a discrete surface is an
analogue of a linear, first--order elliptic differential operator acting on
the space of sections
\[
\Gamma(W)=\prod_{p\in M} W_p
\]
of the bundle $W$. A first order linear operator acts on the ``1--jet'' of a
section by which we mean the collection of its restrictions to the triangles.
Ellipticity is encoded by taking only half of the 1--jet, namely the
restrictions to black triangles.  If we had taken all triangles, we would have
defined a connection on the bundle $W$ which encodes both a holomorphic and an
anti--holomorphic structure.  Thus, a holomorphic structure is given by
assigning to every black triangle $b$ a 2--dimensional linear subspace
\begin{equation}\label{def:hol}
  U_b\subset \Gamma(W_{|b})= \prod_{p\in b} W_p 
\end{equation}
of local holomorphic sections.  A section
$\psi\in\Gamma(W)$ is {\em holomorphic} if $\psi_{|b}\in U_b$ for each black
triangle $b$ of $M$. The space of holomorphic sections is denoted by $H^0(W)$.

An important example of a holomorphic line bundle over a discrete surface $M$
is obtained by any map $f\colon M\to \F\P^1$, where $\F\P^1$ denotes the
Riemann sphere $S^2=\C\P^1$ or the conformal $4$--sphere $S^4=\H\P^1$.  If
$L\subset V$ is the pull--back by $f$ of the tautological bundle over
$\F\P^1$, that is $L_p=f(p)\subset \F^2$, where $V$ is the trivial
$\F^2$--bundle over $M$, then the line bundle $V/L$ has a unique holomorphic
structure such that the constant sections of $V$ project to holomorphic
sections of $V/L$: the linear space $U_b\subset \Gamma(V/L_{|b})$ prescribed
on a black triangle $b$ is the image of the canonical projection
\begin{equation}\label{eq:ind_hol}
  \pi \colon \F^2\rightarrow  \prod_{p\in b} \F^2/L_p =  \Gamma(V/L_{|b}).
\end{equation}
For $U_b$ to be $2$--dimensional we need that the restriction of $f$ to each
black triangle is non--constant.

The space of holomorphic sections of $V/L$ contains the $2$--dimensional
linear system $H\subset H^0(V/L)$ obtained by projection of all constant
sections of $V$.  As in the smooth theory, this linear system $H$ determines
the original map $f$ up to M\"obius transformations as a ratio of independent
holomorphic sections in $H$.  This last is an instance of the Kodaira
correspondence: given a holomorphic line bundle $W$ over a discrete surface
$M$ and a 2--dimensional linear system $H\subset H^0(W)$ without base points,
we obtain a map $f\colon M\to \P H$ whose value at a vertex $p\in M$ is the
line $L_p\subset H$ of sections $\psi\in H$ vanishing at $p$.  Thus, $W\cong
H/L$ with $H$ denoting the trivial $H$--bundle over $M$ and where the
holomorphic structure on $H/L$ is the one induced by the map $f$.

The notion of discrete holomorphicity proposed by Dynnikov and Novikov
\cite{DN97,DN03} for maps into $\C$ is equivalent to the complex case of our
theory.  Our point of view based on the Kodaira correspondence immediately
reveals the M\"obius invariance of discrete holomorphicity, a fact not readily
visible in \cite{DN97,DN03}.

Our exposition of discrete holomorphic geometry is very much in the spirit of
the smooth quaternionic holomorphic geometry developed in
\cite{PP98,BFLPP02,FLPP01,BLPP,BPP}. Several of its key concepts and formulae
carry over verbatim to the discrete setting. The approach based on holomorphic
line bundles and Kodaira correspondence emphasizes the similarity between the
complex and quaternionic case, that is between holomorphic maps into
$S^2=\CP^1$ and conformal surface theory in $S^4=\HP^1$.

As the main application of discrete holomorphic geometry given in this paper
we shed light on the relationship between discrete surface theory in $S^4$ and
integrable systems which is well--established in the smooth situation
\cite{Ko96,Ta97,Ko00,BPP01,Boh03}.  Central to our discussion of the
integrable system aspect of conformal surface theory is the notion of Darboux
transformation for discrete surfaces in the conformal 4--sphere $S^4$.
Iterating this Darboux transformation gives rise to a discrete flow and, if
the underlying surface is a discrete torus, the space of all Darboux
transforms is parameterized by a Riemann surface of finite genus, the spectral
curve of the discrete torus in $\F\P^1$.

Darboux transforms appear, just like in the smooth case, as prolongations of
holomorphic sections: let $f\colon M\to \F\P^1$ be a map whose restriction to
each black triangle is non--constant.  By \eqref{eq:ind_hol} any holomorphic
section $\psi\in H^0(V/L)$ of the induced holomorphic line bundle $V/L$ over
$M$ has a unique ``lift'' $\hat \psi\colon M'\to\F^2$ defined on the set $M'$
of black triangles such that $\pi\hat \psi_b=\psi_{|b}$.  Of course, $\hat
\psi$ is constant if $\psi\in H$ is contained in the linear system $H\subset
H^0(V/L)$ induced by $f$. On the other hand, if $\psi$ is a holomorphic
section with monodromy $h\in\Hom(\Gamma,\F_*)$, that is $\gamma^*\psi=\psi
h_{\gamma}$ for every deck transformation $\gamma\in\Gamma$, then $\hat \psi$
has the same monodromy $h$ and thus, generally, is not constant. Provided that
$\psi$ is not identically zero on any of the black triangles, we obtain a
well--defined map $f^\sharp\colon M'\to \F\P^1$ given by $f^\sharp(b)=\hat
\psi_b \F$ which we call a {\em Darboux transform} of $f$.
  
From this description we see that the Darboux transforms of a discrete surface
arise from solutions to a linear equation.  Alternatively, we can characterize
the Darboux transformation by a non--linear, M\"obius invariant zero curvature
relation which can be expressed as a multi--ratio condition.  This condition
already appeared in a number of instances, including the characterization of
integrable triangular lattices investigated by Bobenko, Hoffmann and Suris
\cite{BHS02} and in the space and time discrete versions of the
Kadomtsev--Petviashvili (KP) and Davey--Stewartson (DS) equation introduced by
Konopelchenko and Schief \cite{KS01,KS03}. From our point of view the former
correspond to 3--periodic sequences of Darboux transforms and the latter to a
sequence of iterated Darboux transforms.  This is consistent with the smooth
theory of conformal surfaces in $S^4$, where the Darboux transformation can be
seen as a time discretization of the Davey--Stewartson flow which approximates
the smooth flow~\cite{Boh03}.

To obtain completely integrable systems requires, like in the smooth case, the
discrete surface $M=T^2$ to be a $2$--torus. We show that the space of all
Darboux transforms $f^{\sharp}$ of a sufficiently generic discrete torus
$f\colon M \rightarrow S^4$ with regular combinatorics can be given the
structure of a compact Riemann surface, the \emph{spectral curve} $\Sigma$ of
$f$. This curve is M\"obius invariant and encodes the constants of motion of
the above mentioned discrete evolution equations. In the smooth case the
spectral curve plays an important role in the study of conformally immersed
tori in $3$-- and $4$--space \cite{Ta98, GS98,BLPP,BPP} and also in the
context of the Willmore problem \cite{S02}.

Since the spectral curve $\Sigma$ parameterizes Darboux transforms of the
discrete torus $f$ there is map
\[
F\colon M' \times \Sigma \rightarrow S^4
\]
assigning to $\sigma\in\Sigma$ the Darboux transform $F(-,\sigma)\colon M'\to
S^4$.  This map has a unique lift $\hat F\colon M' \times \Sigma \rightarrow
\CP^3$ that is holomorphic in $\Sigma$ and satisfies $F=\pi \circ \hat F$
where $\pi\colon \CP^3\rightarrow S^4$ denotes the twistor fibration. Thus, a
discrete torus in $S^4$, which is just a finite set of points with regular
combinatorics, gives rise to a family of algebraic curves parameterized over
the black triangles $b\in M'$. Since holomorphic curves in $\CP^3$ project to
super conformal Willmore surfaces in $S^4$, we also obtain a $M'$--family of
super conformal Willmore surfaces $ F(b,-)\colon \Sigma\to S^4$.

A fundamental property of the Darboux transformation which lies at the heart
of integrability is Bianchi permutability: given two Darboux transforms
$f^\sharp$ and $f^\flat$ of a discrete surface $f\colon M\to S^4$ with regular
combinatorics, there is a common Darboux transform $\hat f$ of $f^\sharp$ and
$f^\flat$.  Bianchi permutability implies that the spectral curve $\Sigma$ of
$f$ is preserved under the Darboux transformation.  Moreover, it can be used
to show that the family of algebraic curves $\hat F$ in $\CP^3$ ``linearizes''
in the Jacobian of $\Sigma$.

We conclude the paper by introducing a discrete flow on polygons in $S^4$ in
terms of a constant cross--ratio condition on the generated quadrilaterals.
For polygons in $S^2$, a reduction of the $S^4$ case, this cross--ratio
evolution was developed in \cite{BP96,BP99,HMNP01,Pi02}. Reducing this flow to
polygons in $\R^3$ gives the doubly discrete smoke ring flow \cite{Ho1,Ho2,
  PSW} (up to translation of the 3--plane).  By thinking of a polygon as a
discrete {\em thin cylinder} this flow in $S^4$ is given by iterated Darboux
transforms.  For closed polygons we thus have a spectral curve and the polygon
flow linearizes on its Jacobian. The corresponding discrete evolution
equations, $1+1$--reductions of the discrete Davey--Stewartson equation, are
the discrete Korteweg--de Vries (KdV) equation for polygons in $S^2$ and the
discrete non--linear Schr\"odinger (NLS) equation for $\R^3$.

\section{Holomorphic line bundles over discrete surfaces}

The approach to discrete conformality proposed in the present paper is based
on the concept of holomorphic line bundles over discrete surfaces.  The idea
behind this approach is that both the intrinsic and extrinsic conformal and
holomorphic geometry can be encoded in the language of holomorphic line bundles
and linear systems of holomorphic sections.  In the first part of the section
we recall the basic notions of the smooth theory of holomorphic line bundles
over Riemann surfaces, in the second part we develop the discrete
counterparts.  A crucial ingredient in the definition of holomorphic line
bundles over discrete surfaces is the combinatorial structure of a
triangulation with black and white colored faces.  The use of such additional
combinatorial data in the context of discrete holomorphicity appears
previously in Dynnikov and Novikov~\cite{DN97,DN03}, Colin de
Verdi\`ere~\cite{CdV98} and also in unpublished notes by Thurston.

\subsection{Smooth theory (complex version)}
A \emph{holomorphic structure} on a complex line bundle $W$ over a Riemann
surface $M$ is given by a so called $\dbar$--operator, a complex linear,
first--order differential operator
\[ \dbar\colon \Gamma(W)\rightarrow \Gamma(\bar KW) \]
satisfying the Leibniz rule 
\begin{equation}
  \label{eq:complex_leibniz}
  \dbar(\psi f) = \dbar(\psi) f + \psi \dbar(f)
\end{equation}
for all real and therefore, by complex linearity, for all complex functions
$f$.  Here, $\bar KW$ denotes the bundle of 1--forms with values in $W$ that
satisfy $*\omega=-i \omega$ with $*$ denoting the complex structure on $T^*M$.
A section $\psi\in \Gamma(W)$ is called \emph{holomorphic} if $\dbar \psi=0$
and the space of holomorphic sections is denoted by $H^0(W)$.  Holomorphic
sections are thus defined as solutions to a linear, first--order elliptic
partial differential equation.  The equation $\dbar \psi=0$ describing
holomorphicity can be seen as half of an equation $\nabla\psi=0$ describing
parallel sections, because every holomorphic structure can be complemented to
a connection $\nabla=\partial+\dbar$ satisfying
$\dbar=\nabla'':=\frac12(\nabla+i*\nabla)$ by choosing an anti--holomorphic
structure $\partial$.

What we mean by the statement that holomorphic line bundles encode the
intrinsic conformal geometry of Riemann surfaces is that the complex structure
on the surface~$M$ itself may be recovered from the first--order elliptic
operator $\dbar$:

\begin{lemma}\label{lem:complex_elliptic} 
  Let $A\colon \Gamma(W)\rightarrow \Gamma(\tilde W)$ be a complex linear,
  first--order elliptic differential operator between sections of complex line
  bundles $W$ and $\tilde W$ over a surface $M$.  Then there exists a unique
  complex structure on $M$ such that $\tilde W \cong \bar K W$ and $A$ becomes
  a $\dbar$--operator satisfying the Leibniz rule \eqref{eq:complex_leibniz}.
\end{lemma}
\begin{proof}
  The complex structure on $M$ is given by the unique complex structure $*$ on
  the bundle $T^*M$ that makes the symbol $\sigma(A)\colon T^*M \rightarrow
  \Hom(W,\tilde W)$ a complex anti--linear operator which then induces an
  isomorphism $\tilde W\cong \bar K W$.  This definition of $*$ makes sense
  because, by ellipticity of the operator $A$, its symbol $\sigma(A)$ is an
  injective bundle morphism from the real rank~2 vector bundle $T^*M$ to the
  complex line bundle $\Hom(W,\tilde W)$.

  The Leibniz rule obviously holds for constant functions $f$. By definition
  of the symbol and the isomorphism $\tilde W \cong \bar KW$, for every point
  $p\in M$ the Leibniz rule holds for real functions $f$ vanishing at $p$.
  Therefore, it holds for all real functions and, by complex linearity, for
  all complex functions.
\end{proof}

In addition to the complex structure on $M$, a first--order elliptic operator
between line bundles $W$ and $\tilde W \cong \bar K W$ over $M$ defines a
complex holomorphic structure on $W$. This additional data is essential for
encoding the extrinsic geometry of holomorphic maps from the Riemann surface
$M$ into complex projective space $\CP^n$: given a Riemann surface $M$, there
is a 1--1--correspondence between
\begin{itemize}
\item projective equivalence classes of holomorphic curves $f\colon
  M\rightarrow \CP^n$ and
\item isomorphy classes of holomorphic line bundles $W$ on $M$ with
  $n+1$--dimensional linear system $H\subset H^0(W)$ of holomorphic sections
  without basepoints.
\end{itemize}
This correspondence between holomorphic curves and linear systems $H\subset
H^0(W)$ is called \emph{Kodaira correspondence}. Because this version of
Kodaira correspondence is of minor importance for the present paper, we skip
further details and refer the Reader to the literature on complex algebraic
geometry, e.g.~\cite{GH}.  Instead, in the following section we give a
detailed treatment of Kodaira correspondence for conformal immersions into
$S^4=\HP^1$.

\subsection{Smooth theory (quaternionic version)}
We briefly describe now the quaternionic versions of the concepts discussed in
the preceding section.  A detailed introduction to quaternionic holomorphic
line bundles over Riemann surfaces can be found in~\cite{FLPP01}.

Let $W$ be a quaternionic line bundle with complex structure $J\in
\Gamma(\End(W))$ over a Riemann surface $M$.  A \emph{holomorphic structure}
on $W$ is given by a quaternionic linear, first--order differential operator
\[ D \colon \Gamma(W)\rightarrow \Gamma(\bar KW) \]
satisfying the Leibniz rule 
\begin{equation}
  \label{eq:quaternionic_leibniz}
  D (\psi f) = D(\psi) f + (\psi df)'' \qquad \textrm{ with } \qquad (\psi
  df)'' = \frac12( \psi df + J \psi *df)
\end{equation}
for all real and hence quaternionic functions $f$.  A section $\psi\in
\Gamma(W)$ is called \emph{holomorphic} if $D\psi=0$. Thus, as in the complex
case, holomorphic sections are defined by a linear, first--order elliptic
partial differential equation. Moreover, the holomorphicity equation $D
\psi=0$ can again be seen as one half of a parallelity equation $\nabla\psi=0$
for a quaternionic connection $\nabla$ satisfying $D=\nabla''$. The
quaternionic analogue to Lemma~\ref{lem:complex_elliptic} is:

\begin{lemma}\label{lem:quat_elliptic} 
  Let $A \colon \Gamma(W)\rightarrow \Gamma(\tilde W)$ be a quaternionic
  linear, first--order elliptic operator between quaternionic line bundles $W$
  and $\tilde W$ over an (oriented) surface $M$.  Then there are unique
  complex structures on $M$, $W$ and $\tilde W$ such that $\tilde W\cong \bar
  K W$ and $A$ is a quaternionic holomorphic structure satisfying the Leibniz
  rule \eqref{eq:quaternionic_leibniz}.
\end{lemma}
\begin{proof}
  The symbol $\sigma(A)\colon T^*M \rightarrow \Hom(W,\tilde W)$ is a
  injective bundle morphisms from the real rank~2 vector bundle $T^*M$ to the
  real rank~4 bundle $\Hom(W,\tilde W)$. Up to sign, there are unique complex
  structures $J$ and $\tilde J$ on $W$ and $\tilde W$ such that the image of
  $\sigma(A)$ is the rank~2 bundle of elements $B\in \Hom(W,\tilde W)$
  satisfying $\tilde JB=BJ$. Moreover, there is a unique complex structure $*$
  on $T^*M$ compatible with the orientation and a unique choice of $J$ and
  $\tilde J$ such that $\sigma(A)\colon T^*M \rightarrow \Hom(W,\tilde W)$
  satisfies $*\sigma(A)=\tilde J \sigma(A)= \sigma(A) J$.

  As in the proof of Lemma~\ref{lem:complex_elliptic} one can check that $A$
  as an operator from $W$ to $\tilde W\cong \bar K W$ satisfies the Leibniz
  rule \eqref{eq:quaternionic_leibniz}.
\end{proof}

An important application of quaternionic holomorphic line bundles over Riemann
surfaces is the Kodaira correspondence for conformal immersions into the
4--sphere, a 1--1--correspondence between
  \begin{itemize}
  \item M\"obius equivalence classes of immersions $f\colon M\rightarrow
    S^4=\HP^1$ of a smooth surface $M$ into the conformal 4--sphere and
  \item isomorphy classes of quaternionic holomorphic line bundles $W$ over
    $M$ with a $2$--dimensional linear system $H\subset H^0(W)$ without
    Weierstrass points (see below).
\end{itemize}
In order to describe this 1--1--correspondence we identify maps $f\colon
M\rightarrow S^4$ from $M$ into the 4--sphere with line subbundles
$L\subset V$ of the trivial quaternionic $\H^2$--bundle $V$ over~$M$. The
quaternionic holomorphic line bundle corresponding to an immersion is then the
quotient bundle $W=V/L$ equipped with the unique holomorphic structure such
that all projections to $V/L$ of constant sections of $V$ are holomorphic. The
2--dimensional linear system $H\subset H^0(W)$ obtained by projecting all
constant sections has no Weierstrass points in the following sense: for every
$p\in M$, the space of sections $\psi\in H$ that vanish at $p$ is
1--dimensional and the vanishing is of first order (the latter because $f$ is
immersed).  Let
\[ L=
\begin{pmatrix}
  f \\ 1 
\end{pmatrix}\H,
\]
where $f\colon M\rightarrow \H$ is the representation of the immersion in an
affine chart. Then, the holomorphic sections $\psi$ and $\varphi$ obtained by
projecting the first and second basis vector of~$\H^2$ to the quotient line
bundle $V/L\cong \H^2/L$ satisfy
\[ \varphi = - \psi f. \] The affine representation $f$ of the immersion is
thus the quotient of two holomorphic sections in the linear system $H$ and
changing the basis $\psi$, $\varphi$ of $H$ amounts to changing the affine
representation of the immersion by a fractional linear transformations.

Conversely, given a holomorphic line bundle $W$ over $M$ together with a
2--dimensional linear system $H\subset H^0(W)$ that has no Weierstrass points,
the line bundle $L$ defined by $L_p=\ker(\ev_p)$ with $\ev_p\colon H
\rightarrow W_p$ denoting the evaluation at $p$ is a conformal immersion
$f\colon M\rightarrow S^4=\HP^1\cong\P(H)$.  Affine representations of this
immersion $f$ are obtained by taking quotients of holomorphic sections $\psi$,
$\varphi\in H$.

\subsection{Complex and quaternionic holomorphic line bundles over discrete
  surfaces}\label{sec:discrete_hol}
The basis of the discrete holomorphic geometry developed in this paper is the
concept of holomorphic line bundles over discrete surfaces. Our definition of
holomorphic line bundles assumes that the discrete surface is equipped with
the additional combinatorial structure of a bi--colored triangulation.  The
idea to use such combinatorial data in the context of discrete holomorphicity
appears previously in Dynnikov and Novikov~\cite{DN97,DN03} and Colin de
Verdi\`ere~\cite{CdV98}.

\begin{definition}
  A \emph{discrete surface} $M$ is the vertex set $\mathcal{V}$ of a
  triangulation $(\mathcal{V},\mathcal{E},\mathcal{F})$ of an oriented smooth
  surface whose set of faces $\mathcal{F}$ is equipped with a bi--coloring,
  that is, the faces $\mathcal{F}$ of the triangulation are decomposed
  $\mathcal{F}=\mathcal{B} \, \dot \cup \, \mathcal{W}$ into ''black'' and
  ''white'' triangles such that two triangles of the same color never share an
  edge in $\mathcal{E}$.
\end{definition}
The existence of such bi--coloring is equivalent to the property that every
closed ``thick path'' of triangles in~$M$ (i.e., every closed path in the dual
cellular decomposition $M^*$) has even length.  By triangulation we mean a
regular cellular decomposition all of whose faces are triangles, where regular
means that the glueing map on the boundary of each $1$-- and $2$--cell is
injective.  In other words, there are no identifications among the three
vertices and edges of a triangle, an assumption that will be necessary in the
definition of holomorphic line bundles below.  Note that we do not assume
triangulations to be simplicial (``strongly regular'') cellular
decompositions.

A discrete surface $M$ is \emph{compact} if its set of vertices if finite or,
equivalently, if the underlying smooth surface is compact. We call a discrete
surface \emph{connected} or \emph{simply connected} if the underlying smooth
surface is connected or simply connected. Similarly, by \emph{(universal)
  covering} of a discrete surface $M$ we mean the vertex set of the
triangulation induced on a covering (the universal covering) of the underlying
smooth surface.

Before we define holomorphic structures on complex or quaternionic line
bundles over discrete surfaces, recall that in the smooth case holomorphic
structures are given by linear, first--order elliptic differential operators
whose kernels describe the space of holomorphic sections.  The discrete
analogue to linear, first--order differential equations being difference
equation defined on the faces of the triangulation, it is natural to define
discrete holomorphic structures by imposing linear equations on the
restrictions of sections to the faces.  Reflecting the fact that the elliptic
operators defining holomorphic structures in the smooth case are the ``half''
of a connection, i.e., can be written as $\bar K$--part $\nabla''$ of a
connection~$\nabla$, it is natural to define holomorphic structure on line
bundles over discrete surfaces by imposing linear equations on half of the
faces only.

\begin{definition}\label{def:disc_hol}
  Let $W$ be a (complex or quaternionic) line bundle over a discrete
  surface~$M$, that is, over the vertex set of a triangulation of an oriented
  surface with black and white colored faces.  A \emph{holomorphic structure}
  on $W$ is given by assigning to each black triangle $b =\{u,v,w\}\in
  \mathcal{B}$ a $2$-dimensional space of sections $U_b \subset \Gamma
  (W_{|b})$ with the property that a section $\psi \in U_b$ vanishing at two
  of the vertices of $b$ has to vanish at all three of them.  A section $\psi
  \in \Gamma (W)$ is called \emph{holomorphic} if $\psi_{|b} \in U_b$ for
  every black triangle $b\in \mathcal{B}$.  The space of holomorphic sections
  of~$W$ is denoted by $H^0(W)$.
\end{definition}

Similar to the smooth case, there are essentially two different links between
the geometry of immersions of discrete surfaces into 4--space and discrete
quaternionic holomorphic geometry: the M\"obius geometric one via Kodaira
correspondence (introduced in the following section) and a Euclidean one based
on the concept of Weierstrass representation (to be discussed in the
forthcoming paper~\cite{BP}).

\subsection{Kodaira correspondence for immersions of discrete surfaces into
  $S^4$}\label{sec:kodaira}
As in the smooth case, there is a 1--1--correspondence between M\"obius
equivalence classes of immersions $f\colon M \rightarrow S^4= \HP^1$ of a
discrete surface $M$ and certain 2--dimensional linear systems of sections of
a quaternionic holomorphic line bundle over $M$.  Replacing quaternions by
complex numbers yields the analogous correspondence between maps into $\CP^1$
and linear systems of holomorphic sections of complex line bundles.

\begin{definition}
  A map $f\colon M \rightarrow S^4= \HP^1$ from a discrete surface $M$ into
  the 4--sphere~$S^4$ is called an \emph{immersion} if and only if $f_p\neq
  f_q$ for adjacent points $p$, $q\in M$.
\end{definition}

In the following, we identify maps from a discrete surface $M$ into
$S^4$ with the corresponding line subbundles $L\subset V$ of a trivial
$\H^2$--bundle $V$ over $M$.  An immersions is thus a line subbundle $L\subset
V$ with the property that $L_p\neq L_q$ for adjacent points $p$, $q\in M$.

The concept of holomorphic structures on line bundles over discrete surfaces
allows to formulate the \emph{Kodaira correspondence} relating the M\"obius
geometry of immersions to quaternionic holomorphic geometry in close analogy
to the smooth case. Kodaira correspondence is a 1--1--correspondence between
  \begin{itemize}
  \item M\"obius equivalence classes of immersions $f\colon M\rightarrow
    S^4=\HP^1$ of a discrete surface $M$ and
  \item isomorphy classes of quaternionic holomorphic line bundles $W$ over
    $M$ equipped with a $2$--dimensional linear system $H\subset H^0(W)$
    satisfying \eqref{eq:immersed}.
\end{itemize}
The holomorphic line bundle corresponding to an immersion $L\subset V$ of a
discrete surface is obtained by exactly the same construction as in the smooth
case: it is the bundle $W=V/L$ equipped with the unique holomorphic structure
such that constant sections $\hat \psi\in \Gamma(V)$ project to holomorphic
sections $\psi= \pi \hat\psi$ of $W=V/L$. This holomorphic structure is given
by assigning to each black triangle $b=\{u,v,w\}\in \mathcal{B}$ the space of
sections
\[ U_b\subset \Gamma(V/L_{|b}) \cong V/L_u \oplus V/L_v \oplus V/L_w \]
obtained by projection of constant sections $\hat \psi\in \Gamma(V)$.  The
2--dimensional linear system $H\subset H^0(W)$ corresponding to the immersion
$L$ is the space of sections of $W=V/L$ obtained by projection of constant
sections of $V$.  The property of $H\subset H^0(W)$ reflecting the fact that
$f$ is an immersion is
\begin{equation}
  \label{eq:immersed}
  \textrm{if } \psi\in H\subset H^0(W) \textrm{ satisfies } \psi_p=\psi_q=0
  \textrm{
    for adjacent  } p,q \in M  \textrm{ then } \psi =0.
\end{equation}
This condition is the discrete analogue to the property that the
2--dimensional linear system $H\subset H^0(W)$ has no Weierstrass points,
i.e., that sections vanishing to second order at one point have to vanish
identically.

Conversely, a 2--dimensional linear system $H\subset H^0(W)$ of a quaternionic
holomorphic line bundle $W$ over a discrete surface $M$ that satisfies
\eqref{eq:immersed} gives rise to an immersion $L\subset V$ of $M$ into
$\HP^1\cong \P(H)$, where $V$ denotes the trivial $H$--bundle $M\times H$ over
$M$ and $L_p =\ker(\ev_p)$ with $\ev_p\colon H\rightarrow W_p$ the evaluation
of a section at $p\in M$.

In coordinates Kodaira correspondence is written by the same explicit formulae
as in the smooth case (cf.\ \cite{FLPP01,BFLPP02}): every 2--dimensional
linear system $H\subset H^0(W)$ satisfying \eqref{eq:immersed} has a basis of
sections $\psi$, $\varphi\in H$ with $\psi$ nowhere vanishing such that there
is a function $f\colon M\rightarrow \H$ satisfying $\varphi = -\psi f$.  With
respect to such a basis $\psi, \varphi$ the immersion into $\HP^1\cong\P(H)$
obtained from the linear system $H$ via Kodaira correspondence is
\[ L =
\begin{pmatrix}
  f \\ 1
\end{pmatrix}\H. \]
Changing the basis of $H$ amounts to changing $f$ by a fractional linear
transform and therefore corresponds to a M\"obius transform of $S^4$.

In Sections~\ref{sec:darboux} and \ref{sec:spectral} we show how the
holomorphic structure on $V/L$ related to an immersion of a discrete surface
via Kodaira correspondence naturally leads to the concept of Darboux
transformation and spectral curve for immersions of discrete surfaces and
tori.

\subsection{Three generations of cellular decompositions with bi--colored
  faces}\label{sec:three_gen}
Despite the far reaching analogies between discrete and smooth holomorphic
geometry, the discrete theory has an interesting additional aspect which is
not visible in the smooth theory because it vanishes when passing to the
continuum limit: given a discrete surface $M$ or, more generally, the vertex
set $\mathcal{V}$ of a regular cellular decomposition
$(\mathcal{V},\mathcal{E},\mathcal{F})$ of a smooth surface equipped with a
bi--coloring $\mathcal{F}=\mathcal{B}\, \dot \cup\, \mathcal{W}$ of its faces
(with regular meaning that all glueing maps are injective), there are two
other regular cellular decompositions $M'$ and $M''$ of the same underlying
smooth surface which are again equipped with bi--colorings of their faces.

The cellular decomposition $M'$ is obtained by taking $\mathcal{B}$ as the set
of vertices, $\mathcal{W}$ as the set of ``black'' faces, and $\mathcal{V}$ as
the set of ``white'' faces with the following combinatorics: the vertices of a
face $w\in \mathcal{W}$ of $M'$ are the elements $b\in \mathcal{B}$ that touch
$w$ (wrt.~$M$) and two such vertices $b_1$, $b_2$ a connected by an edge of
the face $w$ of $M'$ if the share a vertex $v\in \mathcal{V}$ (wrt.~$M$).
Similarly, the vertices of a face $v\in \mathcal{V}$ of $M'$ are the elements
$b\in \mathcal{B}$ that contain $v$ (wrt.~$M$) and two such vertices $b_1$,
$b_2$ are connected by an edge of $M'$ if they touch a common white faces
$w\in \mathcal{W}$ (wrt.~$M$).

Applying the same construction to the decomposition $M'$ leads to the third
cellular decomposition $M''$.  Thus, by cyclically permuting the roles of
$\mathcal{V}$, $\mathcal{B}$ and $\mathcal{W}$ we obtain three generations
$M$, $M'$ and $M''$ of cellular decomposition of the same surface.  The
sequence of cellular decompositions obtained by successively applying the
above construction is actually three--periodic, because by applying the
construction to $M''$ one gets back to the initial cellular decomposition~$M$:
in order to see this, we introduce yet another cellular decomposition of the
surface, the triangulation whose set of vertices consists of $\mathcal{V}$,
$\mathcal{B}$ and $\mathcal{W}$ and whose edges correspond to all possible
incidences of the cellular decomposition $M$, namely
\begin{itemize}
\item the vertex $v\in \mathcal{V}$ is contained in $b\in \mathcal{B}$ or
  $w\in \mathcal{W}$ (wrt.~$M$) and
\item the faces $b\in \mathcal{B}$ and $w\in \mathcal{W}$ touch along an edge
  (wrt.~$M$).
\end{itemize} 
This triangulation has a tri--coloring of its vertices and is the stellar
subdivision of the cellular composition $M^*$ dual to $M$: the dual cellular
decomposition $M^*$ inherits a bi--coloring of its vertices from the black
and white coloring of $M$'s faces; the third color corresponds to the
additional vertices of the stellar subdivision.

Conversely, every triangulation of a surface with tri--colored set of vertices
gives rise to a cellular decomposition with black and white coloring of its
faces (see Figure~\ref{fig:bw_in_3colored}): one color takes the role of the
vertices, the other two become faces. The edges of the triangulation
correspond to the different types of incidences, i.e., that of a vertex lying
on a face and that of two faces intersecting along an edge.
\begin{figure}[hbt]
  \centering
  \resizebox{9cm}{4.5cm}{\includegraphics{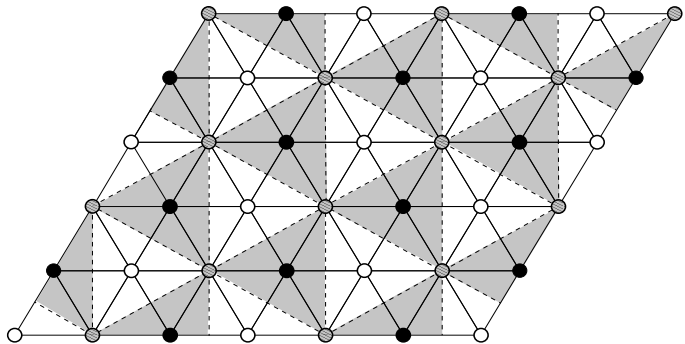}}
  \caption{A triangulation with black and white colored vertices obtained from
    a regular tri--colored triangulation.}
  \label{fig:bw_in_3colored}
\end{figure}

By cyclically permuting the role of the three colors we obtain the sequence
$M$, $M'$ and $M''$ of cellular decompositions with black and white colored
faces which is therefore three--periodic under the above construction, that is
$M=M'''$.

Although the present paper focuses on the case that $M$ is a discrete surface
as defined in Section~\ref{sec:discrete_hol}, i.e., the vertex set of a
triangulation with bi--colored set of faces, discrete holomorphicity can be
defined in the context of more general bi--colored cellular decompositions: in
\cite{BP} we introduce a Weierstrass representation for immersions $f\colon M
\rightarrow \R^4=\H$ defined on the vertex set $M=\mathcal{V}$ of an arbitrary
cellular decomposition with bi--colored faces using ``paired'' holomorphic
line bundles over the faces.  This representation is the direct analogue to
the Weierstrass representation \cite{PP98} for smooth surfaces in $\R^4$.

The assumption that the black faces $\mathcal{B}$ are triangles is necessary
for the Kodaira correspondence between immersions of discrete surfaces and
linear systems of holomorphic line bundles, see Section~\ref{sec:kodaira}.  It
is equivalent to the property that, in the above triangulation with
tri--colored vertices, all black vertices have valence six.

That the white faces $\mathcal{W}$ are triangles is equivalent to six--valence
of the white vertices in the tri--colored triangulation and will be necessary
for the multi ratio characterization \eqref{eq:multiratio_DT} of the Darboux
transformation is Section~\ref{sec:darboux}.

When dealing with iterated Darboux transforms (as in
Sections~\ref{sec:interate_and_KP} and~\ref{sec:bianchi}) we have to assume
that, in addition to $M$, also the cellular decompositions $M'$ and $M''$ have
the property that all their ``black'' faces are triangles or, equivalently,
that all vertices in the above tri--colored triangulation have valence six.
Since $M$ is a triangulation this is equivalent to the six--valence of all
vertices of $M$. This motivates the following definition.

\begin{definition}
  A discrete surface $M$ is said to be of \emph{regular combinatorics} if
  every vertex in the underlying triangulation has valence six.
\end{definition}

If $M$ is a discrete surface with regular combinatorics, the cellular
decompositions $M'$ and $M''$ are also bi--colored triangulations with regular
combinatorics.

Every discrete surface with regular combinatorics is equivalent to the
equilateral triangulation of the plane or a quotient thereof.  In particular,
corresponding to the three edges of triangles in the regular triangulation of
the plane, a discrete surface with regular combinatorics has three
distinguished \emph{directions}. Moreover, for discrete surfaces with regular
combinatorics it makes sense to define a \emph{row} of vertices, black
triangles or white triangles as the vertices that, with respect to the
equilateral triangulation of the plane, are contained on a straight line, or
as the triangles touching a straight line, respectively.

The Euler--formula shows that the only compact discrete surfaces with regular
combinatorics are discrete tori obtained as quotient of the regular
triangulation of the plane by some lattice $\Gamma\cong \Z^2$.  The asymptotic
analysis of spectra of quaternionic holomorphic line bundles over discrete
tori in Section~\ref{sec:spectral} will be considerably simplified by the
assumption that the discrete torus has regular combinatorics.

\section{The Darboux transformation}\label{sec:darboux}

We introduce a Darboux transformation for immersions of discrete surfaces
into~$S^4$.  Like the Darboux transformation for immersions of smooth surfaces
is a time--discrete version of the Davey--Stewartson flow, see \cite{Boh03},
one expects the Darboux transformation of discrete surfaces to be a space--
and time--discrete version of the Davey--Stewartson flow.  This is indeed the
case: for immersed discrete surfaces with regular combinatorics, the Darboux
transformation can be interpreted as the space-- and time--discrete
Davey--Stewartson flow introduced by Konopelchenko and Schief \cite{KS03}. The
case of immersions into $S^2=\CP^1$ is a special reduction of the theory
corresponding to the double discrete KP equation \cite{KS01}.

\subsection{Definition of the Darboux transformation}
The following considerations show how discrete holomorphic geometry, in
particular the Kodaira correspondence for immersions of discrete surfaces,
naturally leads to the Darboux transformation for immersions of discrete
surfaces.

Let $f\colon M \rightarrow S^4=\HP^1$ be an immersion of a simply connected
discrete surface $M$ and denote by $L\subset V$ the corresponding line
subbundle of the trivial $\H^2$--bundle $V$ over $M$.  By definition of the
holomorphic structure on the quaternionic line bundle $V/L$ related to the
immersion via Kodaira correspondence (see Section~\ref{sec:kodaira}), the
restriction $\psi|_{b}$ of a holomorphic section $\psi \in H^0(V/L)$ to the
vertices $v_1$, ..., $v_3$ of a black triangle $b\in \mathcal{B}$ is obtained
by projecting a vector $\hat \psi_b \in \H^2$ to the fibers $(V/L)|_{v_j}$.
Because $\hat \psi_b$ is uniquely determined by $\psi$, every holomorphic
section $\psi \in H^0(V/L)$ gives rise to a section $\hat\psi$ of the trivial
$\H^2$--bundle over the set $\mathcal{B}$ of black triangles.

\begin{definition}
  Let $f\colon M\rightarrow S^4$ be an immersed discrete surface and
  $V/L$ the corresponding quaternionic holomorphic line bundle. A section
  $\hat \psi$ of the trivial $\H^2$--bundle over $\mathcal{B}$ is the
  \emph{prolongation} of a holomorphic section $\psi\in H^0(V/L)$ if and only
  if \[\psi_v=\pi \hat\psi_b\] for every $b\in \mathcal{B}$ and $v\in
  \mathcal{V}$ contained in $b$, where $\pi$ denotes the canonical projection
  to~$V/L$.
\end{definition}

The prolongation $\hat \psi$ of a holomorphic section $\psi$ is constant if
$\psi$ is contained in the 2--dimensional linear system $H\subset H^0(V/L)$
related to $f$ by Kodaira correspondence.  For a holomorphic section $\psi\in
H^0(V/L)$ not contained in $H$, the prolongation $\hat \psi$ is non--constant
and, provided $\psi$ does not vanish identically on any of the black
triangles, $L^\sharp=\hat \psi\H$ defines a map from the set $\mathcal{B}$ of
black triangles into $S^4$.  Exactly as in the smooth case~\cite{BLPP},
Darboux transforms of $L$ will be the maps $L^\sharp$ locally given by
$L^\sharp=\hat \psi\H$ for $\hat \psi$ the prolongation of a nowhere vanishing
local holomorphic section $\psi$ of $V/L$. The non--vanishing of $\psi$ is
reflected in the condition
\begin{equation}
  \label{eq:regular_DT}
  L^\sharp_b\neq L_{v_i} \textrm{ for every black triangle } b=\{v_1,v_2,v_3\}
  \textrm{ and } i=1,...,3.
\end{equation}

We derive now a zero curvature condition characterizing the line subbundles
$L^\sharp$ that locally can be obtained by prolongation from nowhere vanishing
holomorphic sections among all line subbundles with \eqref{eq:regular_DT} of
the trivial $\H^2$--bundle over the black triangles~$\mathcal{B}$. For this we
will in the following consider $\mathcal{B}$ as the vertex set of the cellular
decomposition $M'$ of the underlying surface (see
Section~\ref{sec:three_gen}).  Let $L^\sharp$ be a line subbundle of the
trivial $\H^2$--bundle over the set of vertices $\mathcal{B}$ of $M'$.  If the
bundle $L^\sharp$ satisfies \eqref{eq:regular_DT} it carries a connection
defined by
\begin{equation}\label{eq:connection_DT}
  P_{b_2b_1}\colon L^\sharp_{b_1} \overset{\pi}{\longrightarrow} (V/L)_v
  \overset{\pi^{-1}}{\longrightarrow} L^\sharp_{b_2},
\end{equation}
where $b_1$, $b_2\in \mathcal{B}$ are two black triangles connected by an edge
of $M'$ (i.e., touching a common white triangle), where $v\in \mathcal{V}$ is
the vertex contained in $b_1$ and $b_2$, and where $\pi\colon
L^\sharp_{b_i}\rightarrow (V/L)_v$ are the projections to the quotient bundle.

By definition, a (local) section $\varphi \in \Gamma(L^\sharp)$ is the
prolongation of a holomorphic section of $V/L$ if and only if $\varphi_{b_1}
\equiv \varphi_{b_2} \textrm{ mod } L_v$ for all black triangles $b_1$,
$b_2\in \mathcal{B}$ sharing a vertex $v\in \mathcal{V}$.  This is equivalent to
$\varphi$ being parallel with respect to the connection
\eqref{eq:connection_DT} on the bundle~$L^\sharp$.
The local existence of a non--trivial parallel section $\varphi \in
\Gamma(L^\sharp)$ is equivalent to flatness of the connection
\eqref{eq:connection_DT} on $L^\sharp$.

The curvature of the connection \eqref{eq:connection_DT} on a bundle $L^\sharp$
over $M'$ that satisfies \eqref{eq:regular_DT} is the 2--form assigning to
each face of the cellular decomposition $M'$ the holonomy around that face.
By definition of the connection \eqref{eq:connection_DT} the holonomy is
automatically trivial around the faces of $M'$ that correspond to the vertices
of the original triangulation $M$.  Thus, the only condition for the
connection to be flat is that its holonomy is trivial when going around the
faces of $M'$ that correspond to the white triangles of $M$.  We show now how
to express the triviality of these holonomies as a multi ratio condition. For
this we use the standard embedding $\H\subset \HP^1$ and we write $L$ and
$L^\sharp$ as
\[
L_v=
\begin{pmatrix}
  x_v \\ 1
\end{pmatrix}\H \qquad \textrm{ and } \qquad
L^\sharp_b=
\begin{pmatrix}
  x_b \\ 1
\end{pmatrix}\H
\]
with $x_v$, $x_b\in \H$. The connection \eqref{eq:connection_DT} then becomes
\[  P_{b_jb_i} \begin{pmatrix}
  x_{b_i} \\ 1 \end{pmatrix}= \begin{pmatrix} x_{b_j} \\ 1
\end{pmatrix}(x_{b_j}-x_v)^{-1}(x_{b_i}-x_v)
\]
and, denoting the vertices and black triangles as in
Figure~\ref{fig:combi_DT}, the holonomy around a white triangle (with
clockwise orientation) is
\begin{multline*}
  P_{13,12}P_{12,23}P_{23,13}\begin{pmatrix}
  x_{13} \\ 1
\end{pmatrix}= \begin{pmatrix} x_{13} \\ 1
\end{pmatrix}(x_{13}-x_1)^{-1}(x_{12}-x_1)(x_{12}-x_2)^{-1} \\
(x_{23}-x_2)(x_{23}-x_3)^{-1}(x_{13}-x_3).
\end{multline*}
Hence, the holonomy $P_{13,12}P_{12,23}P_{23,13}$ is trivial if and only if
the multi ratio around the hexagon indicated by the arrows in
Figure~\ref{fig:combi_DT} is
\begin{equation}
  \label{eq:multiratio_DT}
  M_6(x_1,x_{12},x_2,x_{23},x_3,x_{13})=-1,
\end{equation}
where the multi ratio of an ordered six--tuple of quaternions $x_1, ..., x_6\in
\H$ is defined by
\[ M_6(x_1,x_2,x_3,x_4,x_5,x_6) =
(x_1-x_2)(x_2-x_3)^{-1}(x_3-x_4)(x_4-x_5)^{-1}(x_5-x_6)(x_6-x_1)^{-1}. \]
\begin{figure}[hbt]
  \centering
  \resizebox{5cm}{4cm}{\includegraphics{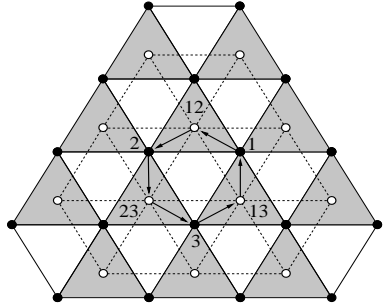}}
  \caption{Combinatorics of the Darboux transformation.}
  \label{fig:combi_DT}
\end{figure}

Note that (unlike suggested by Figure~\ref{fig:combi_DT}) we do not assume
here that the vertices of $M$ have valence six: around the faces of $M'$ that
correspond to vertices of $M$ (and which are not necessarily triangles) the
holonomy of the connection \eqref{eq:connection_DT} is always trivial
(regardless of the valence of the vertex). The triviality of the holonomy
around a face of $M'$ that corresponds to a white triangle of $M$ (and
therefore is a triangle) is always equivalent to the multi ratio condition
\eqref{eq:multiratio_DT}.

The following theorem summarizes the preceding discussion.

\begin{theorem}\label{th:local_DT}
  Let $M$ be a discrete surface, i.e., the vertex set of a black and white
  triangulated smooth surface, and let $M'$ be the cellular decomposition
  whose vertices are the black triangles (see Section~\ref{sec:three_gen}).
  Let $f\colon M\rightarrow S^4$ be an immersion of a discrete surface~$M$
  into $S^4=\HP^1$ and let $f^\sharp \colon M' \rightarrow S^4$ be a map from
  $M'$ to $S^4$ such that the corresponding line subbundles $L$ and $L^\sharp$
  of the trivial $\H^2$--bundles over $M$ and $M'$ respectively
  satisfy~\eqref{eq:regular_DT}.  Then, the following conditions are
  equivalent:
\begin{itemize}
\item[i)] around each white triangle of $M$, the multi ratio condition
  \eqref{eq:multiratio_DT} is satisfied for the hexagon indicated in Figure
  \ref{fig:combi_DT},
\item[ii)] the connection \eqref{eq:connection_DT} on the bundle $L^\sharp$ is
  flat, and
\item[iii)] locally the bundle $L^\sharp$ can be obtained by prolonging a
  holomorphic section of $V/L$ (which needs to be nowhere vanishing since we
  assume \eqref{eq:regular_DT} to be satisfied).
\end{itemize}
\end{theorem}

\begin{definition}\label{def:DT}
  A \emph{Darboux transform} of an immersion $f\colon M\rightarrow S^4$ of a
  discrete surface~$M$ into $S^4=\HP^1$ is a map $f^\sharp \colon M'
  \rightarrow S^4$ with the property that the corresponding line subbundles
  $L$ and $L^\sharp$ of the trivial $\H^2$--bundles over $M$ and $M'$
  satisfy~\eqref{eq:regular_DT} and the conditions of the preceding theorem.
\end{definition}

The situation is completely analogous to the smooth case~\cite{BLPP}:
quaternionic holomorphic geometry provides a correspondence between solutions
to a non--linear zero curvature equation of M\"obius geometric origin on the
one hand and solution to the linear equation describing quaternionic
holomorphicity on the other hand. Locally, one direction of this
correspondence is realized by projecting parallel sections of $L^\sharp$ to the
quotient $V/L$; the other direction is realized by prolonging nowhere
vanishing holomorphic sections of~$V/L$.

It should be noted that the Darboux transformation for surfaces in $\HP^1$ is
directed: as one can easily see from i) of Theorem~\ref{th:local_DT}, if
$f^\sharp$ is an immersed Darboux transform of $f$, then $f^\perp$ is a Darboux
transform of $(f^\sharp)^\perp$, where $\perp$ denotes the dual immersion into
$(\HP^1)^*$.  For any isomorphism $\HP^1\cong (\HP^1)^*$, the transformation
$\perp$ is an orientation reversing conformal diffeomorphism. In particular,
for the right identification the transformation $\perp$ corresponds to
quaternionic conjugation.

\subsection{Example}

As an example we briefly discuss a special case of our theory.  Let $W$ be the
complex holomorphic line bundle over the vertices of the regular triangulation
of the plane all of whose holomorphic sections are of the form $\psi=\varphi
f$, where $\varphi$ is a trivializing holomorphic section and $f$ is a complex
function that maps all black triangles to positive equilateral triangles in
$\C$. Figures~\ref{fig:equilateral_black_triangles} shows two maps $f_i$,
$i=1$, $2$ that are obtained as quotient of the form $f_i=\psi_i/\varphi$ from
local holomorphic sections $\psi_i$, $i=1$, $2$ of $W$.

\begin{figure}[hbt]
  \centering
  \resizebox{12cm}{!}{\includegraphics{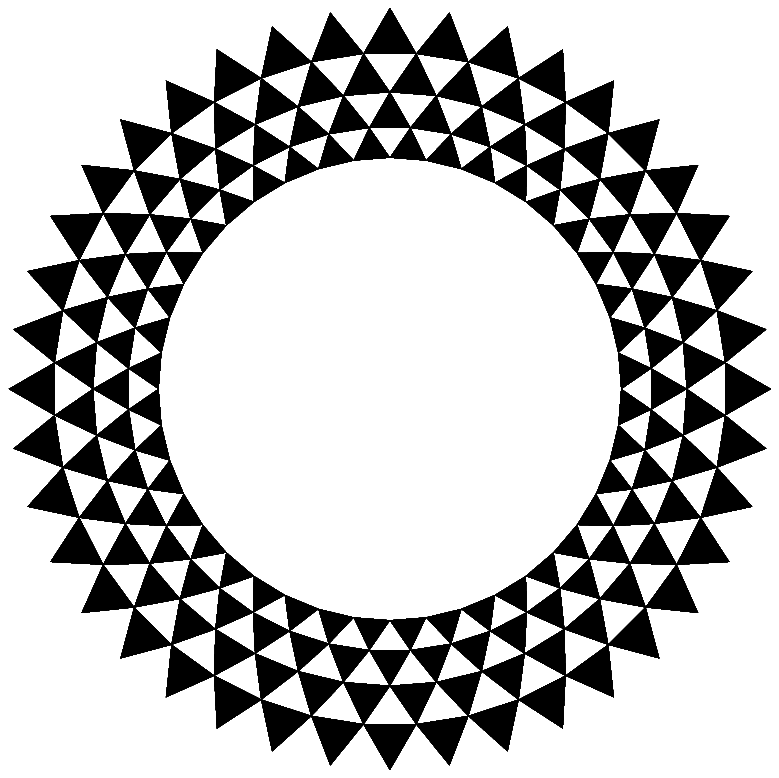}\qquad 
    \includegraphics{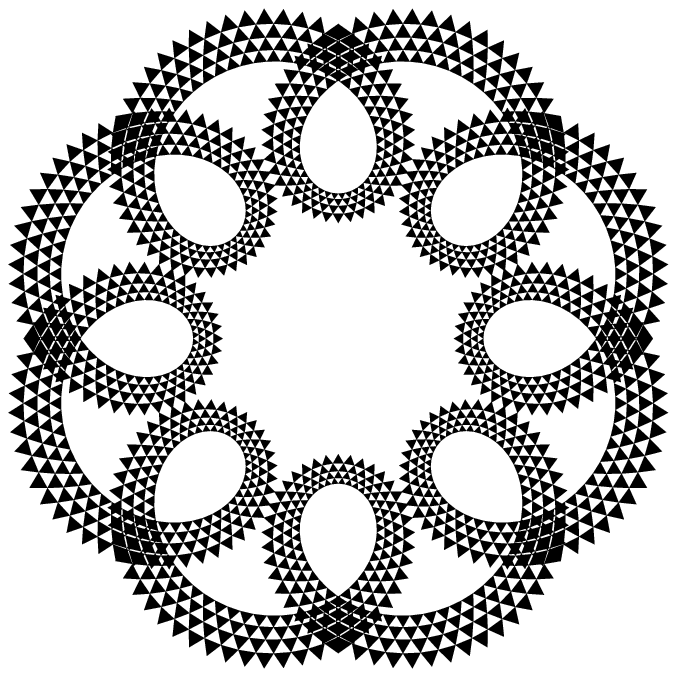}}
  \caption{Two maps $f_1$ and $f_2$ into the plane with equilateral black
    triangles.}
  \label{fig:equilateral_black_triangles}
\end{figure}

Figure~\ref{fig:example_DT} shows the Darboux transform of $f_1$ obtained by
prolonging the holomorphic section $\psi_2=\varphi f_2$, where $f_1$ here is
the 5--fold covering of the ``circle'' and $4$ is the maximal number of ``loops''
that can be ``added''.

\begin{figure}[hbt]
  \centering
  \resizebox{7cm}{!}{\includegraphics{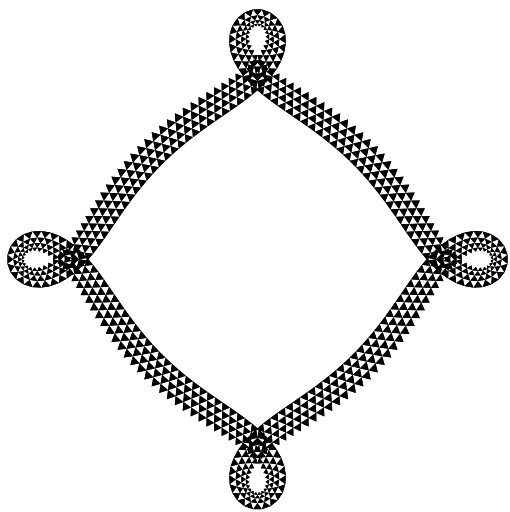}}
  \caption{A Darboux transform of $f_1$.}
  \label{fig:example_DT}
\end{figure}

\subsection{Iterated Darboux transforms and Davey--Stewartson flow}
\label{sec:interate_and_KP}
We show now how a solution to the discrete Davey--Stewartson equation of
Konopelchenko and Schief~\cite{KS03} can be interpreted as a sequence of
iterated Darboux transforms.

For relating iterated Darboux transforms to the discrete Davey--Stewartson
equation it is sufficient to reinterpret the underlying combinatorics,
because Darboux transforms and the space-- and time--discrete
Davey--Stewartson flow are both defined by the same kind of multi ratio
condition: a map $x\colon \Z^3 \rightarrow \H$ a solution to the discrete
Davey--Stewartson equation of \cite{KS03} if and only if it satisfies the
multi ratio equation
\begin{equation}
  \label{eq:multiratio_DS}
  M_6(x_1,x_{12},x_2,x_{23},x_3,x_{13})=-1
\end{equation}
on each cube of the $\Z^3$--lattice with notation as indicated in
Figure~\ref{fig:cube_DS}.

\begin{figure}[hbt]
  \centering
  \resizebox{3.5cm}{!}{\input{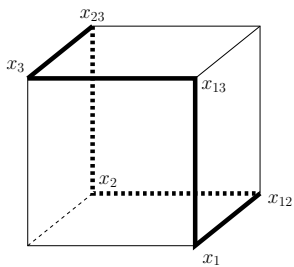}}
  \caption{The space-- and time--discrete Davey--Stewartson flow is defined by
    the multi ratio condition $M_6(x_1,x_{12},x_2,x_{23},x_3,x_{13})=-1$ on the
    cubes of a $\Z^3$--lattice.}
  \label{fig:cube_DS}
\end{figure}

In order to interpret a solution to equation \eqref{eq:multiratio_DS} on the
$\Z^3$--lattice as a sequence of iterated Darboux transforms, we project the
$\Z^3$--lattice to the plane perpendicular to the vector $(1,1,1)$.  This
yields a regular triangulation of the plane whose set of vertices has a
natural tri--coloring: two vertices of the triangulation are given the same
color if they have preimages with the same distance from the plane. This is
illustrated in Figure~\ref{fig:cube_proj} which shows the image of one cube of
the $\Z^3$--lattice under the projection to the plane and the tri--coloring of
its vertices.

\begin{figure}[hbt]
  \centering
  \resizebox{3cm}{2.7cm}{\includegraphics{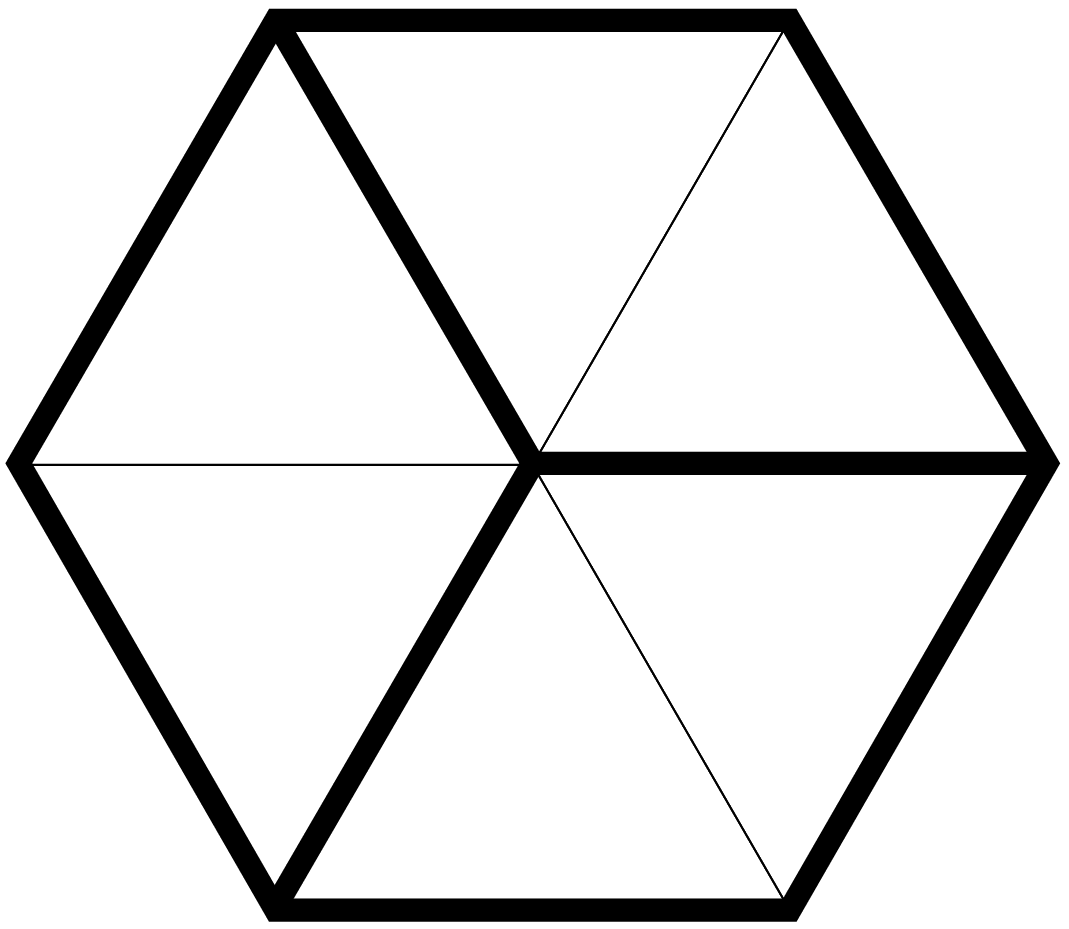}} \qquad \qquad
  \resizebox{3cm}{2.7cm}{\includegraphics{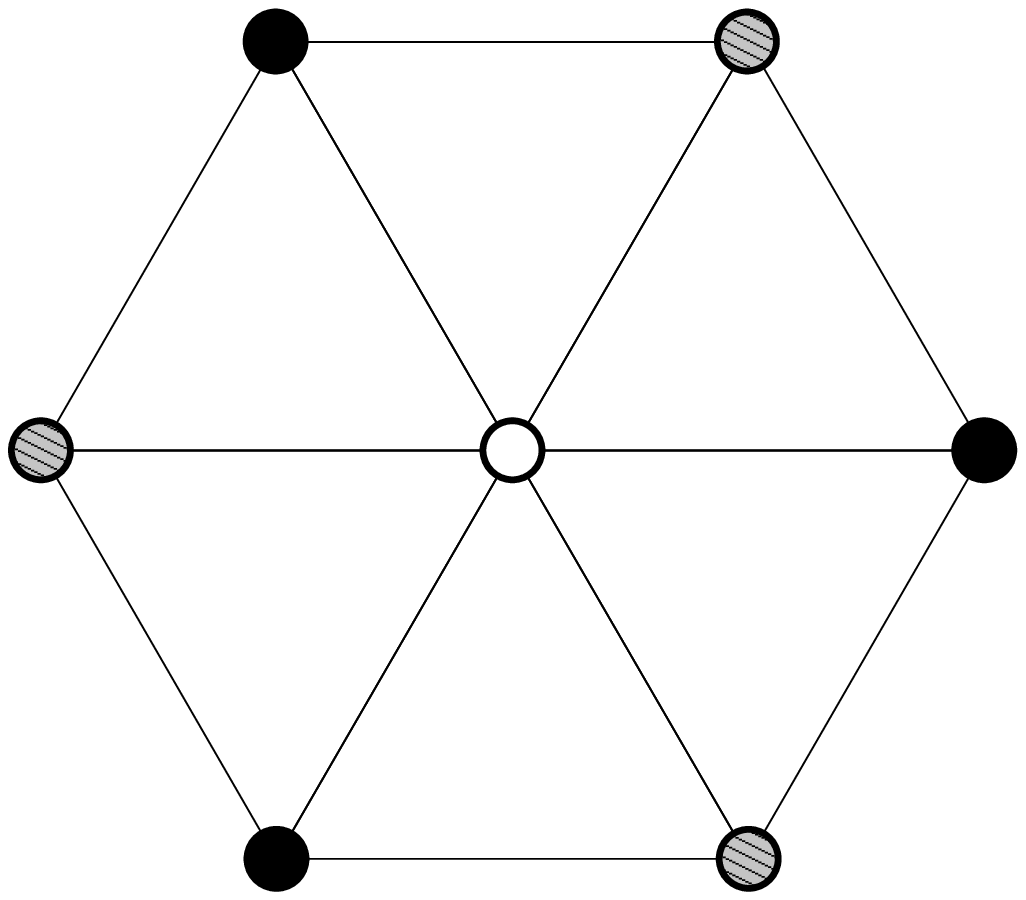}}
  \caption{Planar projection of a cube and the tri--coloring of its vertices.}
  \label{fig:cube_proj}
\end{figure}

As explained in Section~\ref{sec:three_gen}, such a triangulation with
tri--colored set of vertices gives rise to a three periodic sequence $M$, $M'$
and $M''$ of cellular decompositions with black and white colored faces.  In
our case $M$, $M'$ and $M''$ are triangulations: each of the colors can be
taken as the set of vertices of a triangulation whose black triangles
correspond to the next height--level of $\Z^3$ over the plane and whose white
triangles correspond to two height levels above. (As an example, see
Figure~\ref{fig:bw_in_3colored} in Section~\ref{sec:three_gen} which shows the
triangulation $M$ whose set of vertices~$\mathcal{V}$ are the grey vertices of
the tri--colored triangulation, whose black faces $\mathcal{B}$ are the black
vertices, and whose white faces $\mathcal{W}$ are the white vertices.)

By definition, a Darboux transform of an immersion $f\colon M \rightarrow
S^4$ of a discrete surface $M$ is defined on the ``black triangulation''
$M'$.  Similarly, a Darboux transform of an immersion of $M'$ is defined on
the ``white triangulation'' $M''$ and Darboux transforming an immersion of
$M''$ yields a map defined on the initial triangulation $M$.

It is immediately clear from the multi ratio equations \eqref{eq:multiratio_DT}
and \eqref{eq:multiratio_DS}, that---under the correspondence between the
different height--levels of the $\Z^3$--lattice over the plane perpendicular
to $(1,1,1)$ and the three periodic sequence $M$, $M'$ and $M''$ of
triangulations with back and white colored faces---a solution to the
Davey--Stewartson equation corresponds to a sequence of iterated Darboux
transforms: for every cube of the $\Z^3$--lattice, the multi ratio is taken
along the fat line indicated in Figure~\ref{fig:cube_DS} which projects to the
hexagon bounding Figure~\ref{fig:cube_proj} and, in the triangulations with
bi--colored faces, corresponds to a hexagon as indicated in
Figure~\ref{fig:combi_DT}.

\section{The spectral curve}\label{sec:spectral}

For immersions of compact surfaces, a natural question is whether the space of
all Darboux transforms has any interesting the structure.  In the following we
show how, for generic immersions of a discrete torus with regular
combinatorics, this space can be desingularized to a Riemann surface of
finite genus, the spectral curve of the immersion. This spectral curve has
similar properties as in the smooth case.  In particular, it has a canonical
geometric realization as a family of algebraic curves in $\CP^3$. By
construction the spectral curve does not change under M\"obius transformations
and, as a consequence of Bianchi permutability, it is preserved under the
evolution by Darboux transforms.

\subsection{Definition and properties of the spectral
  curve}\label{sec:def_spectral}
The idea to study the spectral curve as an invariant of immersed tori goes
back to Taimanov~\cite{Ta98} and Grinevich and Schmidt~\cite{GS98}. A M\"obius
invariant approach to the spectral curve of smooth tori in $S^4$ is developed
in our papers~\cite{BLPP,BPP}. In the following we present the analogous
theory for discrete tori.

As in the smooth case, Darboux transforms of an immersed discrete surface
$f\colon M \rightarrow S^4$ correspond to nowhere vanishing holomorphic
sections with monodromy of the quaternionic holomorphic line bundle $V/L$: by
definition, a Darboux transform $f^\sharp\colon M'\rightarrow S^4$ of an
immersion $f\colon M \rightarrow S^4$ is a line subbundle $L^\sharp$ of the
trivial $\H^2$--bundle over $M'$ that satisfies \eqref{eq:regular_DT} and has
the property that the induced connection on $L^\sharp$ is flat.  In case that
the underlying surface has non--trivial topology, parallel sections of this
flat connection on $L^\sharp$ are in general not defined on the discrete
surface $M'$ itself, but are sections with holonomy on its universal covering.
Projecting such a parallel section with holonomy of $L^\sharp$ to the quotient
bundle $V/L$ yields a nowhere vanishing holomorphic section with monodromy
of~$V/L$, that is, a holomorphic section of the pull back of $V/L$ to the
universal covering of $M$ which is equivariant with respect to a
representation $h\in \Hom(\Gamma,\H_*)$ of the group of deck transformations
$\Gamma$ in the sense that
\begin{equation}
  \label{eq:monodromy}
  \gamma^* \psi = \psi h_\gamma
\end{equation}
for all deck transformations $\gamma\in \Gamma$.  Taking conversely a line
bundle $L^\sharp$ spanned by the prolongation of a nowhere vanishing holomorphic
section with monodromy of $V/L$ yields a global Darboux transform $f^\sharp$ of
$f$.  This proves the following proposition.

\begin{prop}\label{prop:global_DT_sections_with_monodromy} 
  There is a bijective correspondence between global Darboux transforms of an
  immersed discrete surface $f\colon M \rightarrow S^4 =\HP^1$ and nowhere
  vanishing holomorphic sections with monodromy of the quaternionic
  holomorphic line bundle $V/L$ up to scale.
\end{prop}

Scaling a holomorphic section $\psi$ with monodromy by a quaternion
$\lambda\in \H_*$ yields the section $\psi\lambda$ with conjugated monodromy
$\lambda^{-1} h \lambda$, but does not change the corresponding Darboux
transform.  The space of Darboux transforms of an immersed discrete surface
$f\colon M \rightarrow S^4$ is thus fibered over the conjugacy classes of
representations $\Hom(\Gamma,\H_*)/\H_*$ that are possible for holomorphic
sections with monodromy of $V/L$.  This suggests, as a first step to
understanding the structure of the space of Darboux transforms, to investigate
the monodromies of holomorphic sections and motivates the following
definition.
\begin{definition}
  The \emph{quaternionic spectrum} of a quaternionic holomorphic line bundle
  $W$ over a discrete surface is the subset $\Spec_\H(W)\subset
  \Hom(\Gamma,\H_*)/\H_*$ of conjugacy classes of quaternionic representations
  of the group $\Gamma$ of deck transformations that arise as multiplier of
  holomorphic sections with monodromy of $W$.
\end{definition}

If the underlying surface $M$ is a discrete torus, its group of deck
transformations~$\Gamma$ is abelian such that every conjugacy class of
representations $\Hom(\Gamma,\H_*)/\H_*$ contains a complex representations
$h\in \Hom(\Gamma,\C_*)$ which is unique up to complex conjugation $h\mapsto
\bar h = j^{-1} h j$.  For the study of Darboux transforms of an immersed
discrete torus it is therefore sufficient to consider complex representations
$h\in \Hom(\Gamma,\C_*)$, because every Darboux transform can be obtained from
a holomorphic section $\psi$ with complex monodromy $h\in \Hom(\Gamma,\C_*)$
of $V/L$. This complex representation $h\in \Hom(\Gamma,\C_*)$ is unique up
to complex conjugation: if $\psi$ is a holomorphic section with monodromy $h$,
the section $\psi j$ has monodromy $\bar h$.

\begin{definition}
  The set of complex representations that are possible as multiplier of
  holomorphic sections with monodromy is called the \emph{spectrum}
\[ \Spec(W) = \{ h\in \Hom(\Gamma,\C_*) \mid \exists\textrm{ holomorphic }
\psi\neq 0 \textrm{ with } \gamma^* \psi = \psi h_\gamma \textrm{ for all }
\gamma\in \Gamma \}
\]
of a quaternionic holomorphic line bundle $W$ over a discrete torus $M$.
\end{definition}
The spectrum $\Spec(W)$ is invariant under the real involution $\rho(h)=\bar
h$ and the quaternionic spectrum is the quotient $\Spec_\H(W) =
\Spec(W)/\rho$.

The spectrum $\Spec(V/L)$ of the quaternionic holomorphic line bundle $V/L$
corresponding to an immersed discrete torus $f\colon M \rightarrow \HP^1$ is
never empty, because by Kodaira correspondence the space of sections with
trivial monodromy is at least quaternionic 2--dimensional so that
$h=1\in\Spec(V/L)$.  For arbitrary quaternionic holomorphic line bundles $W$
over discrete tori one can prove (similar to Part~1 in the proof of
Theorem~\ref{th:spectral_curve_W}) that the spectrum $\Spec(W)\subset
\Hom(\Gamma,\C_*)\cong \C_*\times \C_*$ extends to an algebraic subset of
$\C\times \C$ described by the vanishing of one polynomial function, the
determinant of a polynomial family of square matrices: in the system of
complex linear equations characterizing holomorphic sections with a given
monodromy, the number of variables $2|\mathcal{V}|$ equals the number of
equations $2|\mathcal{B}|$, because, for black and white triangulated tori,
Eulers formula $|\mathcal{V}|-|\mathcal{E}|+|\mathcal{B}|+|\mathcal{W}|=0$
together with $|\mathcal{E}|=3|\mathcal{B}|=3|\mathcal{W}|$ implies
\begin{equation}
  \label{eq:torus_vbw}
  |\mathcal{V}|=|\mathcal{B}|=|\mathcal{W}|
\end{equation} 
with $|\mathcal{V}|$, $|\mathcal{E}|$, $|\mathcal{B}|$ and $|\mathcal{W}|$
denoting the number of vertices, edges, black and white triangles.  Provided
the determinant of the corresponding polynomial family of square matrices is
not constant, the spectrum is thus a 1--dimensional algebraic set and can be
normalized to a Riemann surface of finite genus.

\begin{definition}
  Let $W$ be a quaternionic holomorphic line bundle over a discrete torus $M$
  with 1--dimensional spectrum $\Spec(W)$.  The \emph{spectral curve} of $W$
  is the Riemann surface $\Sigma$ normalizing $\Spec(W)$.  Under the
  normalization map $h\colon \Sigma\rightarrow \Spec(W)$, the involution
  $\rho$ of $\Spec(W)$ lifts to an anti--holomorphic involution $\rho\colon
  \Sigma\rightarrow \Sigma$.

  The spectral curve of an immersed discrete torus $f\colon M \rightarrow
  S^4$ for which the holomorphic line bundle $V/L$ has 1--dimensional
  spectrum $\Spec(V/L)$ is defined as the spectral curve of $V/L$.
\end{definition}

The spectral curve as a geometric invariant of tori was first introduced by
Taimanov, Grinevich and Schmidt \cite{Ta98,GS98} for immersions of smooth tori
into $\R^3$. It is defined using Floquet theory for periodic partial
differential operators, see \cite{Kr89,Ku93}.  The discrete analogue of
Floquet theory used to define the spectrum and spectral curve of quaternionic
holomorphic line bundles over discrete tori can immediately be generalized to
an arbitrary linear (``difference'') operator acting on the sections of a
vector bundle over the vertex set of a cellular decomposition of the torus
(with values in another bundle over the vertex set of another cellular
decomposition). The spectrum of such operators is always an algebraic subset
of $\Hom(\Gamma,\C_*)$. In case it is 1--dimensional one can therefore define
a spectral curve of finite genus by normalizing the spectrum.

In the following we will not pursue this general discussion, but investigate
the spectra $\Spec(W)$ of holomorphic line bundles $W$ over discrete tori $M$
with regular combinatorics.

\begin{theorem}\label{th:spectral_curve_W}
  For a discrete torus $M$ with regular combinatorics, the spectrum $\Spec(W)$
  of a quaternionic holomorphic line bundle $W$ over $M$ is always a
  1--dimensional algebraic set such that the spectral curve $\Sigma$ of $W$ is
  defined. It is the Riemann surface of finite genus that normalizes $h\colon
  \Sigma\rightarrow \Spec(W)$ the spectrum and carries an anti--holomorphic
  involution $\rho \colon \Sigma \rightarrow \Sigma$ with $h\circ \rho=\bar h$
  which extends to the compactification $\bar \Sigma$.

  If $W$ satisfies the genericity assumption \eqref{eq:genericity_assumption}
  below, there is a complex holomorphic line bundle $\mathcal{L}\rightarrow
  \Sigma$ whose fiber $\mathcal{L}_\sigma$ over $\sigma\in \Sigma$ is
  contained in the space of holomorphic sections with monodromy $h^\sigma$ of
  $W$ with equality away from a finite subset of $\Sigma$.  The pullback of
  $\mathcal{L}$ under the anti--holomorphic involution $\rho$ is
  $\rho^*\mathcal{L} = \mathcal{L} j$ with $j$ denoting the multiplication of
  sections by the quaternion~$j$.  In particular, the involution $\rho$ has no
  fixed points.
\end{theorem}

The theorem is a discrete version of Theorem~2.6 in \cite{BPP}.  Its proof is
given in Section~\ref{sec:proof}.  As in \cite{BPP}, the main work in the
proof is to investigate the asymptotics of $\Spec(W)$ and the asymptotics for
large monodromies of its holomorphic sections with monodromy. 

There are several essential differences to the smooth case~\cite{BPP} when the
spectral curve has always one or two ends, at most two connected components,
the minimal dimension of the spaces of holomorphic sections with monodromy
$h\in \Spec(W)$ is one, and generically the spectral genus is infinite. For a
quaternionic holomorphic line bundle $W$ over a discrete torus $M$ with
regular combinatorics,
\begin{itemize}
\item the spectral curve always has finite genus,
\item the spectral curve has many ends and may have many connected components
  (cf.\ Corollary~\ref{cor:number_of_infinities} and
  Lemma~\ref{lem:one_dim_analytic}) and
\item it is possible that the minimal dimension of the spaces of holomorphic
  sections with monodromy $h\in \Spec(W)$ is greater than one (see
  Section~\ref{sec:small}).
\end{itemize}

For example, a ``homogeneous'' torus with regular combinatorics in $S^4$, the
orbit of a finite group of M\"obius transformations, has a connected spectral
curve and three pairs of ends which correspond to the three directions of the
triangulation. This shows that, in the space of immersed discrete tori in
$S^4$ with regular combinatorics, there is a Zariski open subset of tori with
irreducible spectral curve.

If $W=V/L$ is the quaternionic holomorphic line bundle induced by a
sufficiently generic immersed torus $f\colon M \rightarrow S^4$, the line
bundle $\mathcal{L}\rightarrow \Sigma$ of Theorem~\ref{th:spectral_curve_W}
allows to geometrically interpret the points of the spectral curve $\Sigma$ as
Darboux transforms of $f$.  For all but finitely many points $h\in \Spec(V/L)$
in the spectrum there is then a unique Darboux transform corresponding to
holomorphic sections with monodromy $h$.  The quotient $\Sigma/\rho$ of the
spectral curve $\Sigma$ by the involution $\rho$ can thus be thought of as a
parameter space for the Darboux transforms of $f$.  More precisely:

\begin{theorem}\label{th:spectral_curve_f}
  Let $f\colon M \rightarrow S^4$ be an immersion of a discrete torus $M$ with
  regular combinatorics for which the bundle $V/L$ satisfies the genericity
  assumption \eqref{eq:genericity_assumption} and has irreducible spectral
  curve.  Taking prolongations of elements in the fibers of the bundle
  $\mathcal{L}$ (see Theorem~\ref{th:spectral_curve_W}) yields a map
  \[\hat F\colon M' \times \Sigma \rightarrow \CP^3\] 
  which is holomorphic in the second variable. The composition $F=\pi \circ
  \hat F$ with the twistor projection $\pi \colon \CP^3\rightarrow \HP^1$ has
  the property that $f^\sharp=F(-,\sigma)\colon M' \rightarrow \HP^1$, for all
  but finitely many points $\sigma \in \Sigma$, is the unique Darboux
  transform of $f$ that belongs to holomorphic section with monodromy
  $h^\sigma$ or $\bar h^\sigma=h^{\rho \sigma}$ of $V/L$.  In particular,
  $\hat F$ is compatible with the fixed point free anti--holomorphic
  involution $\rho$ in the sense that
  \[\hat F(b,\rho \sigma) = \hat F(b,\sigma) j\] for all $b\in M'=\mathcal{B}$
  and all $\sigma\in \Sigma$ with $j$ denoting multiplication by the
  quaternion $j$ seen as an anti--holomorphic involution of $\CP^3$.

  The map $\hat F$ uniquely extends to $M' \times \bar \Sigma$ as a family of
  algebraic curves $\hat F\colon M' \times \bar \Sigma\rightarrow \CP^3$. For
  every row of black triangles of $M$ there is a unique pair of points at
  infinity $\infty$, $\rho\infty\in \bar \Sigma\backslash\Sigma$ such that for
  each $b$ in that row
  \begin{equation}
    \label{eq:f_from_F}
    f(v) = F(b,\infty)=F(b,\rho\infty),
  \end{equation}
  where $v$ denotes the upper vertex of $b$.  Conversely, every point at
  infinity $\infty \in \bar \Sigma \backslash \Sigma$ belongs to one row of
  black triangles with the property that the value of $f$ at the upper vertex
  $v$ of each black triangle $b$ in that row is given by \eqref{eq:f_from_F}.
\end{theorem}

The theorem is a discrete version of Theorem~4.2 in \cite{BLPP} and will be
proven in Section~\ref{sec:proof}.  

\begin{rem}\label{rem:after_main_th}
  Let $f\colon M \rightarrow S^4$ be an immersion of a discrete torus $M$ with
  regular combinatorics that satisfies \eqref{eq:genericity_assumption} and
  has irreducible spectral curve. Then:
  \begin{enumerate}
  \item[i)] For fixed $\sigma \in \Sigma$, the map $f^\sharp\colon
    M'\rightarrow \HP^1$ defined by $f^\sharp=F(-,\sigma)$ is a Darboux
    transform of $f$, except $\sigma$ is one of the finitely many $\sigma\in
    \Sigma$ for which the elements in $\mathcal{L}_\sigma$ are holomorphic
    sections with monodromy of $V/L$ that have a zero.  The maps $f^\sharp$
    obtained for $\sigma$ contained in this finite subset of $\Sigma$ are
    called \emph{singular Darboux transforms}. Note that, over black triangles
    $b\in M'=\mathcal{B}$ on which the elements of $\mathcal{L}_\sigma$ vanish
    identically, such singular Darboux transform $f^\sharp$ is not obtained as
    a prolongation of elements in $\mathcal{L}_\sigma$, but as the limit of
    prolongations of elements in $\mathcal{L}_{\sigma'}$ for
    $\sigma'\rightarrow \sigma$.
  \item[ii)] By construction, a Darboux transform $f^\sharp$ of $f$ can be
    obtained as $f^\sharp = F_\sigma$ for some $\sigma\in \Sigma$ unless the
    nowhere vanishing holomorphic sections with monodromy of $V/L$
    corresponding to $f^\sharp$ belong to one of the finitely many monodromies
    for which the space of holomorphic sections is higher dimensional.
  \item[iii)] For fixed $b\in M'=\mathcal{B}$, the map $F_b=F(b,-)$ is a
    (possibly branched) super--conformal Willmore immersion
    \[F_b\colon \bar \Sigma\rightarrow \HP^1.\] For every vertex $v$ of the
    triangle $b$, there is a pair of points at infinity $\infty$,
    $\rho\infty\in \bar \Sigma\backslash\Sigma$ such that
    \[f(v) = F_b(\infty)=F_b(\rho\infty).\]
      \end{enumerate}
\end{rem}

\subsection{Proof of Theorems~\ref{th:spectral_curve_W} and
  \ref{th:spectral_curve_f}}\label{sec:proof}
The strategy for proving Theorems~\ref{th:spectral_curve_W} and
\ref{th:spectral_curve_f} is similar to the smooth case \cite{BPP,BLPP}.  As
in the smooth case, the following proposition (cf.\ Proposition~3.1 of
\cite{BPP}) is essential for the proof:

\begin{prop} \label{prop:fredholm} Let $D_x$ be a holomorphic family of
  Fredholm operators depending on a parameter $x\in X$ in a connected complex
  manifold $X$.  Then, the minimal kernel dimension is generic and attained
  away from an analytic subset $Y\subset X$.  If the manifold $X$ is complex
  1--dimensional, the vector bundle $\mathcal{K}_x=\ker(D_x)$ defined over
  $X\backslash Y$ holomorphically extends through the isolated set $Y$ of
  points with higher dimensional kernel.  If the index $\index{(D_x)}$ of the
  operators $D_x$ is zero, the set of $x\in X$ for which $D_x$ is invertible
  is locally given as the vanishing locus of one holomorphic function.
\end{prop}

The proof of Theorem~\ref{th:spectral_curve_W} is divided in two parts: in
Part~1 we define a family $D_h$ of linear operators that depends
holomorphically on $h\in \Hom(\Gamma,\C_*)$ and has the property that the
kernel of $D_h$ is isomorphic to the space of holomorphic sections with
monodromy $h$ of $W$.  In Part~2 we investigate the asymptotics of $\Spec(W)$
under the assumption that the underlying discrete torus has regular
combinatorics. The minimal kernel dimension of the family of operators $D_h$
is then zero and the subset $Y=\Spec(W)$ of $X=\Hom(\Gamma,\C_*)$ is
non--empty and hence 1--dimensional.  It is an algebraic subset given by one
polynomial equation, because $D_h$ is a polynomial family of operators between
finite dimensional spaces of the same dimension.  The spectrum $\Spec(W)$ can
thus be normalized to a spectral curve $\Sigma$ of finite genus.  Applying the
proposition again to $D_h$ seen as a family of operators parametrized over the
1--dimensional manifold $\Sigma$ yields a holomorphic vector bundle
$\mathcal{L}\rightarrow \Sigma$ (possibly with rank depending on the connected
component) such that, for every $\sigma\in \Sigma$, the fiber
$\mathcal{L}_\sigma$ is contained in the space of holomorphic sections with
monodromy $h^\sigma$ with equality away from finitely many points.  Under the
genericity assumption \eqref{eq:genericity_assumption} we show, by
investigating the asymptotics of holomorphic sections for large monodromies,
that the minimal kernel dimension of $D_h$ on each component of $\Sigma$ is
one. The bundle $\mathcal{L}$ is thus a holomorphic line bundle.

Although the family of operators $D_h$ defined in Part~1 of the proof
immediately generalizes to discrete tori with arbitrary combinatorics, the
asymptotic analysis in Part~2 depends essentially on the assumption that the
torus has regular combinatorics: an important ingredient in Part~2 is an
adapted bases of the lattice and a compatible fundamental domain for discrete
tori with regular combinatorics.  This reduces the problem of finding
holomorphic sections with monodromy of $W$ to the study of eigenlines of a
family of operators that is polynomial in one complex variable.
The methods used in Part~2 of the proof of Theorem~\ref{th:spectral_curve_W}
are the main difference to the smooth case~\cite{BPP}, where the
1--dimensionality of $\Spec(W)$ and $\mathcal{L}$ is proven by asymptotic
comparison to the ``vacuum'' case of quaternionic holomorphic line bundles
obtained by doubling complex holomorphic line bundles.

The proof of Theorem~\ref{th:spectral_curve_f} is similar to the smooth
case~\cite{BLPP}: taking prolongations of the elements in the fibers of the
bundle $\mathcal{L}\rightarrow \Sigma$ yields an extrinsic realization of
$\Sigma$ as a family $\hat F\colon M'\times \Sigma \rightarrow \CP^3$ of
algebraic curves parametrized over $M'$.  The asymptotics of sections of
$\mathcal{L}$ for large monodromies shows that $\hat F$ extends to a family of
algebraic curves.

\textbf{Proof of Theorem~\ref{th:spectral_curve_W} -- Part~1:} For
quaternionic holomorphic line bundles $W$ over discrete tori with regular
combinatorics we define now a holomorphic family of operators $D_h$ depending
on $h\in \Hom(\Gamma,\C_*)$ with the property that the kernel of $D_h$ is
isomorphic to the space of holomorphic sections with monodromy $h$ of~$W$.
Recall that a discrete torus $M$ has regular combinatorics if all its vertices
have valence six or, equivalently, if $M$ is the quotient of the black and
white colored, equilateral triangulation of the plane by some lattice
$\Gamma\cong \Z^2$.

\begin{lemma}\label{lem:adapted_basis}
  Let $M$ be a discrete torus with regular combinatorics.  For each direction
  in the regular triangulation of the plane there is a positive basis
  $\gamma$, $\eta$ of the lattice $\Gamma$ with the property that $\gamma$
  is parallel to the edges corresponding to the direction and that the black
  triangles touching $\gamma$ lie on its positive side, see
  Figure~\ref{fig:good_basis}.
\end{lemma}

\begin{figure}[hbt]
  \centering
  \resizebox{10cm}{!}{\input{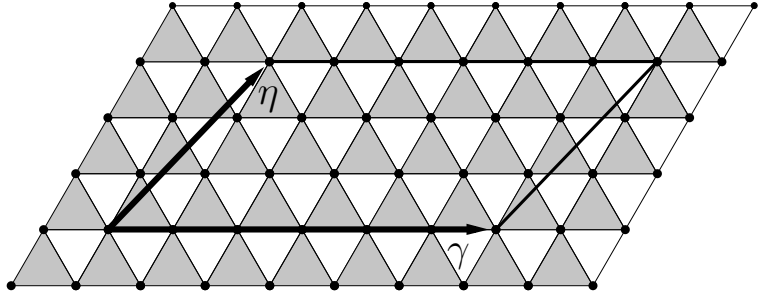}}
  \caption{An adapted basis of the lattice.}
  \label{fig:good_basis}
\end{figure}

\begin{proof} Because $M$ has only finitely many points, for each direction of
  the triangulation there is a unique smallest $\gamma\in \Gamma$ that is
  parallel to the corresponding edges of the triangles and has the property
  that black triangles touching $\gamma$ lie on its positive side. With
  respect to an arbitrary basis $\tilde \gamma$, $\tilde \eta$ the
  vector $\gamma$ takes the form $\gamma= \tilde \gamma a + \tilde
  \eta b$ with $a$, $b\in \Z$ relatively prime. Hence, there are $c$,
  $d\in \Z$ with $ad-bc=1$ such that $\gamma$ together with $\eta =
  \tilde\gamma c+ \tilde\eta d$ form a positive basis of the lattice.
  This proves that each of the three possible $\gamma$ can be complemented
  to a positive basis $\gamma$, $\eta$ of the lattice with $\eta$
  unique up to adding multiples of $\gamma$.
\end{proof}

\begin{definition}\label{def:adapted_basis}
  A basis $\gamma$, $\eta$ of the lattice with the properties of
  Lemma~\ref{lem:adapted_basis} is called \emph{adapted}.  An adapted basis is
  called \emph{normalized} if the angle between $\gamma$ and $\eta$ is smaller
  or equal $\pi/3$ and the angle between $\gamma$ and $\eta-\gamma$ is greater
  than $\pi/3$.  We call the number of points of a torus in the direction of
  $\gamma$ the \emph{length} $n$ of the torus with respect to the adapted
  basis $\gamma$, $\eta$ and the number of rows above $\gamma$ its
  \emph{thickness} $m$.
\end{definition}

The choice of an adapted basis $\gamma$, $\eta\in \Gamma$ fixes an isomorphism
$h\mapsto (h(\gamma),h(\eta))$ from $\Hom(\Gamma,\C_*)$ to $\C_*\times \C_*$.
This introduces coordinates $(\mu,\lambda)\in \C_*\times \C_*$ on
$\Hom(\Gamma,\C_*)$, where $\mu$ and $\lambda$ are the monodromies in the
$\gamma$ and $\eta$ directions, that is $\mu = h(\gamma)$ and $\lambda =
h(\eta)$.

In order to define the family of operators we fix now a normalized adapted
basis $\gamma$, $\eta$ of the lattice~$\Gamma$. Moreover, by choosing a vertex
$v_0$ of the universal covering of $M$, we fix a \emph{compatible fundamental
  domain} of $M$ consisting of all vertices of the regular triangulation that
are of the form $v_0 + t_1\gamma + t_2 \eta$ with $t_i\in [0,1[$ as indicated
by the fat black points in Figure~\ref{fig:def_spec}.
\begin{figure}[hbt]
  \centering
  \resizebox{10cm}{!}{\input{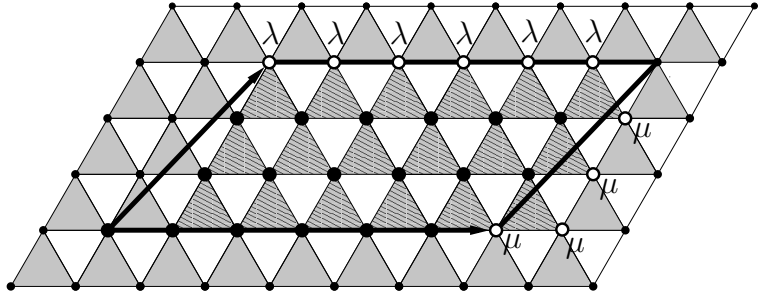}}
  \caption{Fundamental domain of a discrete torus $M$ with regular
    combinatorics.}
  \label{fig:def_spec}
\end{figure}

To reflect the dependence on these choices, our family of operators is denoted
by $D_{\mu,\lambda}$, where $(\mu,\lambda)\in \C_*\times \C_*$ are the
coordinates on $\Hom(\Gamma,\C_*)$ introduced by $\gamma$, $\eta$. The
operator $D_{\mu,\lambda}$ is defined on the direct sum of the lines $W$
over the fat black dots in Figure~\ref{fig:def_spec} (i.e., over the vertices
in the fundamental domain) and takes values in the trivial $\H$--bundle over
the shaded black triangles.  It is the composition of
\begin{itemize}
\item the complex (but not quaternionic) linear operator that maps a section
  of $W$ defined over the fat black dots to its unique extension as a section
  with monodromy corresponding to $(\mu,\lambda)\in \C_*\times \C_*\cong
  \Hom(\Gamma,\C_*)$ and
\item a non--trivial choice (fixed independent of $(\mu,\lambda)\in \C_*\times
  \C_*$) of quaternionic linear equations $W_p\oplus W_q \oplus
  W_r\rightarrow \H$ defining holomorphicity on the shaded black triangles
  $b=\{p,q,r\}$.
\end{itemize}
By definition of $D_{\mu,\lambda}$ we have:
\begin{lemma} \label{lem:fiber_of_L} The kernel of $D_{\mu,\lambda}$ is
  isomorphic to the space of holomorphic sections with monodromy $h$ of $W$,
  where $h\in \Hom(\Gamma,\C_*)$ is the multiplier whose coordinates with
  respect to the chosen adapted basis $\gamma$, $\eta$ are $(\mu,\lambda)\in
  \C_*\times \C_*$, i.e., $\mu=h(\gamma)$ and $\lambda=h(\eta)$.
\end{lemma}

The family $D_{\mu,\lambda}$ of complex linear operators is polynomial (of
order one) in $\mu$ and $\lambda$, because it only involves the extension of
the section over the fat black points in Figure~\ref{fig:def_spec} to the fat
white points (which is obtained by multiplication with $\mu$ or $\lambda$ from
the values of the section at $\Gamma$--related fat black points, as indicated
by the $\mu$'s and $\lambda$'s in Figure~\ref{fig:def_spec}).  The operators
$D_{\mu,\lambda}$ have index $\Ind{(D_{\mu,\lambda})}=0$, because they are
operators between finite dimensional complex vector spaces of the same
dimension $2|\mathcal{V}|=2|\mathcal{B}|$, see~\eqref{eq:torus_vbw}.  In the
coordinates $(\mu,\lambda)\in \C_*\times \C_*$, the spectrum is thus an
algebraic set $\Spec(W)\subset \Hom(\Gamma,\C_*)$ given as the vanishing locus
of one polynomial function, the determinant of the family of operators
$D_{\mu,\lambda}$.

The fact that the determinant of the family $D_{\mu,\lambda}$ is polynomial is
an essential difference to the smooth theory where one has to deal with a
holomorphic family of Fredholm operators whose determinant is transcendental.
As a consequence, the spectral curve of immersed discrete tori is always a
Riemann surface of finite genus while for immersions of smooth tori it can
have infinite genus.

\textbf{Proof of Theorem~\ref{th:spectral_curve_W} -- Part~2:} We investigate
now the asymptotics of the spectrum $\Spec(W)$ and of the spaces of
holomorphic sections for large monodromies $h\in \Spec(W)$.  Compared to the
smooth case, the ``asymptotic analysis'' is simplified significantly by the
fact that the polynomial family of operators $D_{\mu,\lambda}$ depending on
$(\mu,\lambda)\in \C_*\times \C_*$ immediately extends to $\C\times \C$.  For
understanding the asymptotics it turns out to be sufficient to study this
extension along the line $(\mu,\lambda)\in \C\times \{0\}$.


Because the restriction of a holomorphic section to a black triangle is
uniquely determined by its value on two of the vertices, a section in the
kernel of $D_{\mu,\lambda}$ is uniquely determined by its values on the lowest
row of fat black points in Figure~\ref{fig:def_spec}. It is helpful to see the
black triangles touching $\gamma$ as arrows pointing in the propagation
direction for the ``evolution'' of ``initial data'' of holomorphic sections.
In fact, the values of a holomorphic section in $\ker(D_{\mu,\lambda})$ over
the fat black points in Figure~\ref{fig:def_spec} can be obtained recursively
from the values on the lowest row of fat black points by successive extension
to the row above: for fixed $\mu\in \C$, there is a complex linear operator
$T_0(\mu)$ mapping the initial data given on the lowest row (together with its
extension to the white points in the lowest row by multiplication with $\mu$)
to the direct sum of the fibers of $W$ over the fat black points in the row
above such that the resulting section on the vertices of the lowest row of
shaded black triangles is holomorphic.  There is an operator $T_1(\mu)$
mapping this data again to the row above etc. Finally, for $m$ the thickness
of the torus with respect to the adapted basis $\gamma$, $\eta$, there is an
operator $T_{m-1}(\mu)$ mapping the upper row of fat black points to the row
of white points and then, under the identification with respect to the lattice
vector $\eta$, to the lowest row of fat black points. Hence, for every $\mu\in
\C$ we obtain a complex linear endomorphism
\[ H(\mu)= T_{m-1}(\mu)\cdot T_{m-2}(\mu)\cdot\cdot\cdot T_1(\mu)\cdot
T_0(\mu) \] of the direct sum of the fibers of $W$ over the fat black points
in the lowest row of Figure~\ref{fig:def_spec}. By construction, the operator
$D_{\mu,\lambda}$ has a non--trivial kernel if and only if $\lambda$ is an
eigenvalue of the endomorphism $H(\mu)$.  The spectrum of $W$ is thus given by
the set of $(\lambda,\mu)\in \C_*\times \C_*$ for which
$\det(\lambda-H(\mu))=0$. In particular, it is a branched covering of the
$\mu$--plane $\C_*$ with the points corresponding to eigenvalue $\lambda=0$
removed.  This proves the following lemma.
\begin{lemma}\label{lem:one_dim_analytic}
  If $W$ is quaternionic holomorphic line bundle over a discrete torus $M$
  with regular combinatorics, the algebraic subset $\Spec(W)\subset
  \Hom(\Gamma,\C_*)\cong \C_*\times \C_*$ is 1--dimensional such that it is
  possible to define the spectral curve $\Sigma$ of $W$ as the normalization
  of its spectrum $\Spec(W)$. The number of connected components of $\Sigma$
  is at most $2\cdot n_{max}$, where $n_{max}$ denotes the maximal length of
  $M$, cf.\ Definition~\ref{def:adapted_basis}.
\end{lemma}

The following two lemmas are essential for understanding the asymptotic
behavior of holomorphic sections near the ends $\bar \Sigma\backslash \Sigma$
of $\Sigma$.

\begin{lemma}\label{lem:infinity}
  Let $\Sigma$ be the spectral curve of a quaternionic holomorphic line bundle
  over a discrete torus with regular combinatorics.  For each end $\infty\in
  \bar\Sigma \backslash\Sigma$ of the spectral curve $\Sigma$, there is a
  unique normalized adapted basis $\gamma$, $\eta$ of the lattice $\Gamma$ for
  which the meromorphic functions $\mu = h(\gamma)$ and $\lambda=h(\eta)$
  satisfy
  \[ \mu(\infty)\in \C_* \qquad \textrm{ and } \qquad \lambda(\infty)=0.\]
  With respect to the other two normalized adapted bases, the corresponding
  meromorphic function satisfies $\mu(\infty)=\infty$ and $\mu(\infty)=0$,
  respectively.
\end{lemma}
\begin{proof} Denote by $\gamma_i$, $\eta_i$, $i=1$,...,$3$ the three possible
  normalized adapted bases with numbering as in Figure~\ref{fig:three_basis}.
  \begin{figure}[hbt]
    \centering \resizebox{3cm}{!}{\input{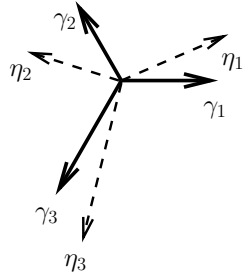}}
    \caption{The three normalized adapted bases $\gamma_i$, $\eta_i$,
      $i=1,...,3$ of the lattice.}
  \label{fig:three_basis}
  \end{figure}

  There are $l_1$, ..., $l_3\in \N$ such that $l_1\gamma_1+ l_2 \gamma_2 +
  l_3\gamma_3=0$. Therefore, the meromorphic functions $\mu_i=h(\gamma_i)$ on
  $\Sigma$ describing the monodromies in direction of $\gamma_i$ satisfy
  \begin{equation}
    \label{eq:three_monodromies}
    \mu_1^{l_1}\cdot \mu_2^{l_2} \cdot \mu_3^{l_3}  = 1.
  \end{equation}
  Because $\infty$ is a point at infinity, not all $\mu_i(\infty)$,
  $i=1$,...,$3$ can be in $\C_*$ and, by \eqref{eq:three_monodromies}, at
  least one of the $\mu_i(\infty)$, $i=1$,...,$3$ has to be zero and one has
  to be infinity.  Without loss of generality we can assume that
  $\mu_1(\infty)=\infty$ and
\begin{itemize}
\item $\mu_2(\infty)\in \C_*$ (Case~1) or
\item $\mu_2(\infty)=0$ (Case~2).
\end{itemize}
In Case~1, we obtain that $\lambda_2(\infty)=0$ and $\mu_3(\infty)=0$ while
$\lambda_3(\infty)=0$ or $\lambda_3(\infty)\in \C_*$ (depending on whether the
angle between $\gamma_3$ and $\eta_3$ is smaller or equal $\pi/3$).  This
proves the statement for Case~1.

To complete the proof we show that Case~2 is impossible as it implies
$\lambda_2(\infty)=0$ . This is obvious if the angle between $\gamma_2$ and
$\eta_2$ is $\pi/3$, because then the operator $T_0(0)$ from Part~2 defined
with respect to $\gamma_2$, $\eta_2$ has no kernel and thus $D_{0,0}$ has no
kernel.  If the angle between $\gamma_2$ and $\eta_2$ is smaller than $\pi/3$,
we pass to the discrete torus $\tilde M$ obtained as quotient of the regular
triangulation by $\tilde \Gamma=\Span_\Z\{ \gamma_1,\gamma_2\}$ for which
$\tilde \gamma_2= \gamma_2$, $\tilde \eta_2=-\gamma_1$ is a normalized adapted
basis.  Because holomorphic sections with monodromy of $W$ give rise to
holomorphic sections with monodromy of the pullback $\tilde W$ of $W$ to the
covering $\tilde M$ of $M$, the spectral curve $\Sigma$ of $W$ embeds into
that of $\tilde W$.  With respect to $\tilde \gamma_1$ and $\tilde \eta_2$,
the image of the point $\infty$ under this embedding has coordinates $0=\tilde
\mu = \tilde h(\tilde \gamma_1)$ and $0=\tilde \lambda = \tilde h(\tilde
\eta_2)$.  But this is impossible by the above argument, because the angle
between $\tilde \gamma_1$ and $\tilde \eta_2$ is~$\pi/3$.
\end{proof}

\begin{lemma} \label{lem:number_points} Let $W$ be a quaternionic holomorphic
  line bundle over a discrete torus $M$ with regular combinatorics. Counted
  with multiplicities there are $2m$ points $(\mu,0)\in \C_*\times \C$ for
  which $D_{\mu,0}$ has a non--trivial kernel, where $m$ is the thickness of
  the torus with respect to the adapted basis chosen to define
  $D_{\mu,\lambda}$.
\end{lemma}
\begin{proof}
  The operator $D_{\mu,0}$ has a non--trivial kernel if and only if one of the
  operators $T_i(\mu)$, $i=0$, ..., $m-1$ has a non--trivial kernel.  For
  $\mu=0$, only $T_0(\mu)$ can have a kernel: the sections with support in the
  fat black points of the lowest row that are not contained in one of the
  shaded black triangles.  For $\mu\in \C_*$, the operator $T_i(\mu)$ has a
  non--trivial kernel if there is a holomorphic section defined on the
  vertices of the $i^{th}$ row of shaded black triangles that vanishes on all
  upper vertices of the triangles and has horizontal monodromy~$\mu$.  Because
  on every black triangle there is a quaternionic 1--dimensional space of
  holomorphic sections vanishing on the upper vertex, such holomorphic section
  is parallel with respect to a quaternionic connection on the restriction of
  $W$ to the vertices in the $i^{th}$ row and $\mu$, $\bar \mu$ are the
  complex eigenvalues of the quaternionic holonomy of this connection, in the
  following called \emph{horizontal holonomies}. In other words, the operator
  $T_i(\mu)$ with $\mu\in \C_*$ is invertible if and only if $\mu$ is not one
  of the horizontal holonomies of the quaternionic connection on the
  restriction of $W$ to the $i^{th}$ row of fat black points induced by the
  holomorphic structure.
\end{proof}

Lemma~\ref{lem:infinity} shows that the ends $\bar \Sigma\backslash\Sigma$ of
$\Sigma$ are divided into three different types corresponding to the three
different directions of the regular triangulation on the torus: for each of
the three normalized adapted bases $\gamma$, $\eta$ one type of ends
$\infty\in \bar \Sigma\backslash\Sigma$ is characterized by the property that
the meromorphic functions $\mu=h(\gamma)$ and $\lambda=h(\gamma)$ take values
$\mu(\infty)\in \C_*$ and $\lambda(\infty)=0$ while the other two types are
characterized by $\mu(\infty)=0$ and $\mu(\infty)=\infty$, respectively.  The
meromorphic function $\mu$ is a branched covering (see above) and hence
non--constant on the components of $\bar \Sigma$ such that each component
contains at least one end of every type.

The family $D_{\mu,\lambda}$, $(\mu,\lambda)\in\C_*\times \C_*$ of operators
defined by the choice of an adapted basis and a compatible fundamental domain
extends to $(\mu,\lambda)\in \C\times \C$ and the first type of ends is mapped
to those $(\mu,0)\in \C_*\times\{0\}$ for which $D_{\mu,0}$ has a non--trivial
kernel.  Under the \textbf{genericity assumption} that the horizontal
holonomies (see the proof of Lemma~\ref{lem:number_points})
\begin{equation}
  \label{eq:genericity_assumption}
  \mu_0, \bar \mu_0, ...,
  \mu_{m-1}, \bar\mu_{m-1} \textrm{ are mutually distinct for every normalized
    adapted basis,}
\end{equation}
the horizontal holonomies are in 1--1--correspondence to the ends of the
respective type. This shows:

\begin{cor}\label{cor:number_of_infinities}
  Let $\Sigma$ be the spectral curve of a holomorphic line bundle $W$ over a
  discrete torus with regular combinatorics for which
  \eqref{eq:genericity_assumption} is satisfied.  The number of ends
  $\bar\Sigma\backslash\Sigma$ of $\Sigma$ is $2(m_1+m_2+m_3)$, where $m_i$
  denotes the thickness of the torus with respect to the normalized adapted
  basis $\gamma_i$, $\eta_i$. Every connected component of $\bar \Sigma$
  contains at least three ends, one corresponding to every direction of the
  triangulation.
\end{cor}

The assumption \eqref{eq:genericity_assumption} implies that the kernel of
$D_{\mu,0}$ is 1--dimensional if $\mu$ is one of the horizontal holonomies. In
particular, by Proposition~\ref{prop:fredholm}, the minimal and therefore
generic kernel dimension of $D_{\mu,\lambda}$ is one on every connected
component when $D_{\mu,\lambda}$ is seen as a family of operators parametrized
over $\Sigma' = \{\sigma \in \bar \Sigma\mid \mu(\sigma)\in \C,
\lambda(\sigma)\in \C \}$. This implies that the kernels of $D_{\mu,\lambda}$
extend through the isolated points in $\Sigma$ at which the kernel is higher
dimensional to a holomorphic line bundle $\mathcal{L}\rightarrow \Sigma$.
Although the family of operators $D_{\mu,\lambda}$ depends on the choice of
adapted basis and fundamental domain, by Lemma~\ref{lem:fiber_of_L} every
fiber $\mathcal{L}_\sigma$ is isomorphic to a space of holomorphic sections
with monodromy $h^\sigma$ of $W$.

\begin{lemma}\label{lem:linebundle} If a quaternionic holomorphic line bundle
  $W$ over a discrete torus with regular combinatorics satisfies
  \eqref{eq:genericity_assumption}, there is a unique complex holomorphic line
  bundle $\mathcal{L}\rightarrow \Sigma$ with the property that the fiber
  $\mathcal{L}_\sigma$ over $\sigma\in \Sigma$ is a subspace of the space of
  holomorphic sections with monodromy $h^\sigma$ of $W$ that, for all but
  finitely many points $\sigma \in \Sigma$, coincides with the space of
  holomorphic sections with the given monodromy.  The pullback of
  $\mathcal{L}$ under the anti--holomorphic involution $\rho\colon
  \Sigma\rightarrow \Sigma$ is $\rho^*\mathcal{L}=\mathcal{L}j$. In
  particular, $\rho$ has no fixed points.
\end{lemma}
\begin{proof}
  The existence and uniqueness of the line bundle $\mathcal{L}$ is explained
  in the preceding discussion.  The number of points with higher dimensional
  kernel is finite because, under the normalization map $h\colon
  \Sigma\rightarrow \Spec(W)$, it is mapped to a subset of the finite set of
  singular points of the algebraic curve $\Spec(W)$.

  For generic $\sigma\in \Sigma$, the spaces of holomorphic section with
  monodromy $h^\sigma$ and $h^{\rho\sigma}$ of $W$ are complex 1--dimensional.
  Multiplying a non--trivial element $\psi\in
  \mathcal{L}_\sigma\backslash\{0\}$ by the quaternionic $j$ then yields a
  non--trivial holomorphic section $\psi j$ with monodromy
  $h^{\rho\sigma}=\bar h^\sigma$ which spans $\mathcal{L}_{\rho\sigma}$.  This
  implies $\rho^*\mathcal{L}=\mathcal{L}j$. In particular, although complex
  conjugation on $\Spec(W)$ leaves all real representations fixed, its lift to
  the anti--holomorphic involution $\rho\colon \Sigma\rightarrow \Sigma$ has
  no fixed points.
\end{proof}

\textbf{Proof of Theorem~\ref{th:spectral_curve_f}:} The idea behind the
definition of $F(-,\sigma)$ and $\hat F(-,\sigma)$ is essentially the same as
in the smooth case~\cite{BLPP}.  For generic immersions $f\colon M \rightarrow
S^4$ of discrete tori with regular combinatorics, we show that all but
finitely many points $\sigma\in \Sigma$ of the spectral curve give rise to a
1--dimensional space of holomorphic sections with monodromy $h^\sigma$ of
$V/L$ which are nowhere vanishing and hence define a unique Darboux transform
of~$f$.  The maps $\hat F$ and $F$ are then defined by the projective lines
obtained from prolonged non--trivial elements in the fibers of $\mathcal{L}$,
where $\mathcal{L}\rightarrow \Sigma$ is the holomorphic line bundle whose
fiber $\mathcal{L}_\sigma$ over $\sigma$ is a subspace of the holomorphic
sections with monodromy $h^\sigma$ of $V/L$.

In contrast to the smooth case, for immersions of discrete tori with regular
combinatorics it is not always true that the generic dimension of the spaces
of holomorphic sections with monodromy is one, see Section~\ref{sec:small}.
In Theorem~\ref{th:spectral_curve_f} we therefore restrict to immersions with
the property that $W=V/L$ satisfies the genericity assumption
\eqref{eq:genericity_assumption}. This assures the 1--dimensionality of
$\mathcal{L}\rightarrow \Sigma$ and, in case the spectral curve is
irreducible, implies that non--trivial elements in a generic fiber of
$\mathcal{L}$ are nowhere vanishing holomorphic sections with monodromy of
$V/L$ (which is important for defining $\hat F$ and $F$).  Moreover, the
assumption \eqref{eq:genericity_assumption} enables us to derive the
asymptotics of $F$.

We first define $\hat F$ on $M' \times \Sigma$.  By
Lemma~\ref{lem:linebundle}, away from a finite set $S_1\subset \Sigma$ of
points at which the space of holomorphic sections with the corresponding
monodromy $h^\sigma$ is higher dimensional, the fiber $\mathcal{L}_\sigma$ of
the holomorphic line bundle $\mathcal{L}\rightarrow \Sigma$ coincides with the
space of holomorphic section with monodromy $h^\sigma$.

\begin{lemma}\label{lem:generically_non_vanishing}
  Let $\Sigma$ be the spectral curve of an immersed discrete torus with
  \eqref{eq:genericity_assumption}. If $\Sigma$ is irreducible, there is a
  discrete set $S_2\subset \Sigma$ such that the non--trivial elements of
  $\mathcal{L}_\sigma$ for $\sigma \in \Sigma \backslash S_2$ are nowhere
  vanishing holomorphic sections with monodromy $h^\sigma$ of $V/L$.
\end{lemma}

Lemma~\ref{lem:generically_non_vanishing} (and
Theorem~\ref{th:spectral_curve_f}) holds more generally under the assumption
that the quotient $\Sigma/\rho$ of the spectral curve under the fixed point
free anti--holomorphic involution $\rho$ is connected, i.e., that $\Sigma$ is
either connected or the direct sum of two connected Riemann surfaces
interchanged under $\rho$. The latter happens for generic immersed tori in
$\CP^1$, cf.\ Lemma~\ref{lem:two--sphere}. It would be interesting to know
whether the assumption that $\Sigma/\rho$ is connected can be dropped from
Lemma~\ref{lem:generically_non_vanishing} and
Theorem~\ref{th:spectral_curve_f}.

\begin{proof}
  We choose a normalized adapted basis of the lattice $\Gamma$ and a
  compatible fundamental domain.  Such choice defines a family
  $D_{\mu,\lambda}$ of operators whose kernels describe the holomorphic
  sections with monodromy, see Lemma~\ref{lem:fiber_of_L}.

  Assume $S_2$ does not exist.  Then, because $\Sigma$ is connected, there is
  a vertex $v_0$ in the fundamental domain (i.e., a fat black dot in
  Figure~\ref{fig:def_spec}) that is a zero for all holomorphic sections
  in~$\mathcal{L}_\sigma$, $\sigma\in \Sigma$.  But this
  contradicts~\eqref{eq:genericity_assumption}: assume the compatible
  fundamental domain of the torus is chosen such that the vertex $v_0$ is
  contained in the upper row of fat black dots in Figure~\ref{fig:def_spec}.
  Denote $\infty\in \bar\Sigma\backslash\Sigma$ one of the two ends for which
  $\mu(\infty)\in \C_*$ is one of the two horizontal holonomies corresponding
  to the upper row of black triangles.  Then $D_{\mu(\infty),0}$ has a
  1--dimensional kernel whose elements are sections that vanish identically on
  the upper row of white dots in Figure~\ref{fig:def_spec} and therefore have
  no zeroes on the upper row of fat black dots.  This contradicts the
  assumption that all elements in~$\mathcal{L}_\sigma$, $\sigma\in \Sigma$
  vanish at~$v_0$.
\end{proof}

A local holomorphic section of $\mathcal{L}$ is a complex holomorphic family
$\sigma \mapsto \psi^\sigma$ of quaternionic holomorphic sections with
monodromy $h^\sigma$ of $V/L$.  With respect to a fundamental domain
compatible with a normalized adapted basis, for every shaded black triangle
$b$ in Figure~\ref{fig:def_spec}, the prolongation $\sigma \mapsto \hat
\psi_b^\sigma$ is a holomorphic map to $\C^4=(\H^2,\i)$.  Projectively, this
yields a map $\hat F\colon M' \times \Sigma \rightarrow \CP^3$ which is
holomorphic in the second variable: if $\sigma \mapsto \hat \psi_b^\sigma$
vanishes at some point $\tilde \sigma \in S_2$ (that is, if $\psi^{\tilde
  \sigma}$ vanishes simultaneously at all three vertices of $b$) we define
$\hat F(b,\tilde \sigma)$ as the line described by the first non--trivial
element in the Taylor expansion of $\sigma \mapsto \hat \psi_b^\sigma$ at
$\tilde \sigma$.

Denote $F=\pi \circ \hat F$ the composition of $\hat F$ with the twistor
projection $\pi\colon \CP^3\rightarrow\HP^1$.  By construction, for all
$\sigma\in \Sigma\backslash S_2$ the map $f^\sharp=F(-,\sigma)$ is a Darboux
transform of $f$. For $\sigma \in S_2$ it is a singular Darboux transform as
defined in i) of Remark~\ref{rem:after_main_th}. Similar to
$\mathcal{L}\rightarrow \Sigma$, the map $\hat F$ clearly does not depend on
the choice of adapted basis and fundamental domain.  The quaternionic symmetry
$\rho^*\mathcal{L}=\mathcal{L}j$ of the bundle $\mathcal{L}$, see
Lemma~\ref{lem:linebundle}, implies
\begin{equation}
  \label{eq:equivariance}
  \hat F(b,\rho\sigma)=\hat F(b,\sigma)j.
\end{equation}

In order to extend $\hat F$ through an end $\infty\in \bar \Sigma \backslash
\Sigma$ we chose the unique normalized adapted basis $\gamma$, $\eta$ of the
lattice with respect to which $\mu(\infty)\in \C_*$ and $\lambda(\infty)=0$,
where $\mu=h(\gamma)$ and $\lambda=h(\eta)$, see Lemma~\ref{lem:infinity}.
For every compatible fundamental domain the family of operators
$D_{\mu,\lambda}$ extends to $\C\times \C$ and, because $\infty$ is a point at
infinity, the operator $D_{\mu(\infty),0}$ has a non--trivial kernel.  Without
loss of generality we can therefore assume that the compatible fundamental
domain is chosen such that the operator $T_{m-1}(\mu(\infty))$ is not
invertible.  We extend the map $\hat F$ to $\infty$ by taking the projective
lines described by the prolongation of a nowhere vanishing local holomorphic
section of $\mathcal{L}'$ defined near $\infty$, where
$\mathcal{L}'\rightarrow \Sigma'$ denotes the kernel bundle of
$D_{\mu,\lambda}$ seen as a family of operators parametrized over $\Sigma' =
\{\sigma \in \bar \Sigma\mid \mu(\sigma)\in \C, \lambda(\sigma)\in \C \}$.  By
our choice of fundamental domain the sections in the 1--dimensional kernel of
$D_{\mu(\infty),0}$ do not vanish on the $(m-1)^{th}$ row of fat black dots in
Figure~\ref{fig:def_spec}, but its holomorphic extension to the (white) dots
in the $m^{th}$ row above vanishes identically. This implies
\[ f(v) = F(b,\infty)=F(b,\rho \infty)\] for $b$ a shaded black triangles $b$
in the upper row and $v$ its upper vertex.  The fact that for every choice of
adapted basis the kernel bundle $\mathcal{L}$ extends to $\Sigma'$ proves that
the points corresponding to singular Darboux transforms cannot accumulate at
infinity. This implies:
\begin{cor}
  The subset $S_2\subset \Sigma$ of points $\sigma\in \Sigma$ for which the
  holomorphic sections in $\mathcal{L}_\sigma$ have zeroes is finite.
\end{cor}

\subsection{Spectral curves of small tori}\label{sec:small}
We discuss the spectral curves of immersions $f\colon M \rightarrow S^4$ of
discrete tori $M$ with regular combinatorics and vertex set consisting of
three or four points.  In particular, we give an example of a discrete torus
in $S^4$ with regular combinatorics for which every $\sigma \in \Sigma$ gives
rise to a 2--dimensional space of holomorphic sections with monodromy
$h^\sigma$ of $V/L$.

By definition, the minimal number of vertices of a discrete surface is three,
because the triangulation of the underlying smooth surface is assumed to be
regular.  Every discrete torus with three vertices has regular combinatorics:
its number of black triangles is three by \eqref{eq:torus_vbw} and, because
every triangle is assumed to have three distinct vertices, every vertex has
valence six.  Because three is a prime number, for every adapted basis the
thickness of the torus (see Definition~\ref{def:adapted_basis}) has to be one.
Therefore, every discrete torus with three vertices is a ``thin torus'' as
shown in Figure~\ref{fig:three_and_four_thin}.

\begin{figure}[hbt]
  \centering

  \resizebox{10cm}{!}{\includegraphics{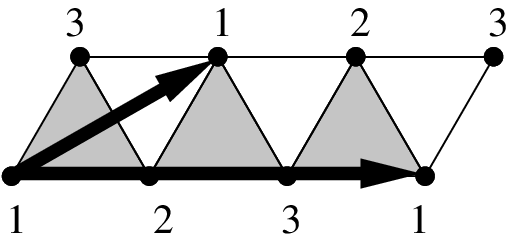}\qquad 
    \includegraphics{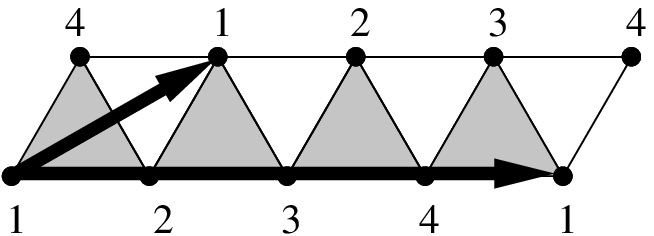}}
  \caption{Thin tori with three and four vertices.}
  \label{fig:three_and_four_thin}
\end{figure}

A discrete torus with four vertices not necessarily has regular combinatorics.
Up to isomorphism, there are two discrete tori with non--regular
combinatorics: the torus obtained by ``adding'' a vertex of valence two to the
discrete torus with three vertices and the torus with two vertices of valence
four and two vertices of valence eight (obtained by gluing opposite edges of a
square that is triangulated by mutually joining the midpoints of all four
edges by straight lines).

\begin{rem}
  The discrete torus $M$ with four vertices one of which has valence two is
  ``isospectral'' to the torus $\hat M$ with three vertices: let more
  generally $M$ be a discrete torus obtained by adding a vertex of valence two
  to another discrete torus $\hat M$. For every quaternionic holomorphic line
  bundle $W$ over $M$ and every monodromy $h$, the space of holomorphic
  section with monodromy $h$ of $W$ is then canonically isomorphic to the
  space of holomorphic sections with the same monodromy of the restriction of
  $W$ to $\hat M$.
\end{rem}

A discrete torus with four points and regular combinatorics can either be a
thin torus as in Figure~\ref{fig:three_and_four_thin}, that is, it admits an
adapted basis for which its thickness is one, or the thickness for every
adapted basis is two (the only prime factor of four) and the torus is of the
form shown in Figure~\ref{fig:fourpoint2}.

\begin{figure}[hbt]
  \centering
  \resizebox{4cm}{!}{\input{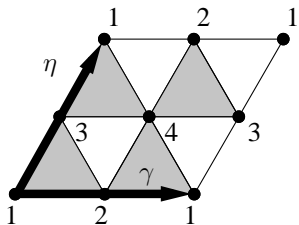}}
  \caption{``Thicker'' torus with four vertices.}
  \label{fig:fourpoint2}
\end{figure}

For immersions of discrete tori with three or four vertexes we can always
assume, after applying a M\"obius transformation, that the immersion takes
values in $\CP^1$.  By the following lemma, an immersion of a discrete torus
with values in $\CP^1$ has a decomposable spectral curves.

\begin{lemma} \label{lem:two--sphere} Let $f\colon M\rightarrow \CP^1\subset
  \HP^1$ be an immersion of a discrete torus $M$ with regular combinatorics
  that takes values in the 2--sphere $S^2=\CP^1$ . Then, the spectral curve
  $\Sigma$ of $f$ can be decomposed $\Sigma=\Sigma_1\, \dot \cup \, \Sigma_2$
  into two Riemann surfaces which are interchanged under the anti--holomorphic
  involution $\rho$.  Moreover, all Darboux transforms $f^\sharp=F(-,\sigma)$
  corresponding to points $\sigma \in \Sigma $ take values in~$\CP^1$.
\end{lemma}
\begin{proof}
  As in Sections~\ref{sec:def_spectral} and~\ref{sec:proof} (with obvious
  adaptations like replacing quaternions by complex number etc.)  one can
  define a spectral curve $\Sigma_1$ for an immersion $f$ with values in
  $\CP^1$ as the normalization of the possible monodromies of holomorphic
  sections of the complex holomorphic line bundle $\hat V /\hat L$ over $M$
  obtained from $f$ by complex Kodaira correspondence, where $\hat V$ is
  trivial $\C^2$--bundle over $M$ and $\hat L$ the line subbundle
  corresponding to~$f$.

  The quaternionic holomorphic line bundle $V/L$ corresponding to $f$ via
  quaternionic Kodaira correspondence is then the quaternionification of $\hat
  V/\hat L$. All vector spaces and operators in Section~\ref{sec:proof} have
  thus a natural direct sum decomposition with respect to this
  quaternionification.  In particular, the $\CP^1$--spectral curve $\Sigma_1$
  of $\hat V /\hat L$ is embedded into the $\HP^1$--spectral curve $\Sigma$ of
  $V/L$ and $\Sigma=\Sigma_1\, \dot \cup \, \Sigma_2$, where $\Sigma_2=\rho
  \Sigma_1 \cong \bar \Sigma_1$ is the image of $\Sigma_1$ under the
  involution $\rho$.
\end{proof}

As explained in Section~\ref{sec:curves}, immersions of thin tori as in
Figure~\ref{fig:three_and_four_thin} are polygons in~$\CP^1$.  By
Remark~\ref{rem:spec_s2} and Theorem~\ref{th:curve_DT_as_cylinder_DT}, the
$\HP^1$--spectral curve $\Sigma$ of such a thin torus in $\CP^1\subset \HP^1$
is the double of the $\CP^1$--spectral curve of the corresponding polygon
in~$\CP^1$. Moreover, by Remark~\ref{rem:spec_s2}, the dimension of the space
of holomorphic sections with monodromy $h^\sigma$ of $V/L$ for generic
$\sigma\in \Sigma$ is one.

In the remainder of this section we investigate the spectral curve of an
arbitrary immersion $f\colon M \rightarrow \CP^1\subset \HP^1$ of the torus
$M$ in Figure~\ref{fig:fourpoint2}.  By Lemma~\ref{lem:two--sphere}, its
$\HP^1$--spectral curve is the double of its $\CP^1$--spectral curve.  It is
therefore sufficient to determine this $\CP^1$--spectral curve.

In the affine coordinate $\C\rightarrow \CP^1\; x\mapsto [x, 1]$ the
immersion is given by four mutually disjoint points $x_1$, ..., $x_4\in\C$. We
trivialize the bundle $\hat V/\hat L$ corresponding to the immersion by the
holomorphic section $\varphi=\pi e_1$ with $\pi$ denoting the canonical
projection to the quotient $\hat V/\hat L$ and $e_1$ denoting the first basis
vector of $\C^2$ seen as a section of the trivial bundle $\hat V$.  Over a
black triangle $b$ with vertices $p$, $q$ and $r$, the holomorphic structure
on $\hat V/\hat L$ is given by the linear form
\begin{equation}
  \label{eq:hol_str_in_trivialization}
  \alpha = \left(\,\frac{x_r-x_q}{x_p-x_q}\,,\,
    \frac{x_r-x_p}{x_q-x_p}\,,\,-1\, \right)
\end{equation}
acting on $\Gamma(\hat V/\hat L)_{|b}=(\hat V/\hat L)_{|p}\oplus (\hat V/\hat
L)_{|q} \oplus (\hat V/\hat L)_{|r}\cong \C \oplus \C \oplus \C$ with
isomorphism induced by~$\varphi$.

Using \eqref{eq:hol_str_in_trivialization} we compute now the operator
$H(\mu)$ defined in Part~2 of the proof of Theorem~\ref{th:spectral_curve_W}
(see Section~\ref{sec:proof}): the operator $T_0(\mu)$ maps a section
$(y_1,y_2) \in \C^2$ over the first and second vertex in the lowest row of
Figure~\ref{fig:fourpoint2} to the section
\[\left(
  \frac{x_3-x_2}{x_1-x_2} y_1 + \frac{x_3-x_1}{x_2-x_1} y_2 \, , \,
  \frac{x_4-x_1}{x_2-x_1} y_2 + \frac{x_4-x_2}{x_1-x_2} y_1 \mu \right) \]
over the first and second vertex in the middle row. Similarly, the operator 
$T_1(\mu)$ maps a section
$(y_3,y_4) \in \C^2$ over the first and second vertex in the middle row of
Figure~\ref{fig:fourpoint2} to the section
\[\left(
  \frac{x_1-x_4}{x_3-x_4} y_3 + \frac{x_1-x_3}{x_4-x_3} y_4 \, , \,
  \frac{x_2-x_3}{x_4-x_3} y_4 + \frac{x_2-x_4}{x_3-x_4} y_3 \mu \right) \]
over the first and second vertex in the lower row. The composition $H(\mu) =
T_1(\mu) \circ T_0(\mu)$ is
\[ 
(y_1,y_2) \mapsto \frac{(x_1-x_3)(x_2-x_4)\mu -
  (x_2-x_3)(x_1-x_4)}{(x_1-x_2)(x_3-x_4)} (y_1,y_2). 
\] 
Hence, for every $\mu\in \C_*$ there is a unique eigenvalue $\lambda$ of
$H(\mu)$ and there is a unique value $\mu\in \C_*$ for which this eigenvalue
$\lambda$ is zero.  The spectral curve $\Sigma$ of $f\colon M \rightarrow
S^4=\HP^1$ is therefore the double of the three--punctured projective plane
$\CP^1\backslash \{\infty_1,\infty_2,\infty_3\}$. Because $\lambda$ is always
a double eigenvalue of $H(\mu)$, the bundle $\mathcal{L}\rightarrow \Sigma$
describing the holomorphic sections with given monodromy of $V/L$ is a rank
two bundle.

\subsection{Bianchi permutability}\label{sec:bianchi}
A Bianchi type permutability theorem usually states something like: given two
transformations $f^\sharp$ and $f^\flat$ of $f$, there is a common
transformation $\tilde f$ of both $f^\sharp$ and $f^\flat$ that can be
computed algebraically.  This may be visualized by
\begin{equation*}
  \xymatrix{ &  f^\flat \ar@{~>}[dr] & \\
    f \ar@{~>}[ur]\ar@{~>}[dr] & & {\tilde f}.\\
    &  f^\sharp \ar@{~>}[ur] &}
\end{equation*}

Recall that iterated Darboux transforms of an immersion $f\colon M\rightarrow
S^4$ of a discrete surface $M$ are only defined if $M$ has regular
combinatorics, because all three cellular decompositions $M$, $M'$ and $M''$
(see Section~\ref{sec:three_gen}) need to be triangulations: while the
original immersions $f$ is defined on $M$, the Darboux transforms $f^\sharp$
and $f^\flat$ are defined on $M'$ and the iterated Darboux transform $\tilde
f$ is defined on $M''$.

\begin{theorem}\label{th:bianchi}
  Let $f\colon M \rightarrow S^4$ be an immersion of a discrete surface
  $M$ with regular combinatorics and let $f^\sharp$, $f^\flat\colon M'
  \rightarrow S^4$ be two immersed Darboux transforms of $f$ with
  $f^\sharp(b)\neq f^\flat(b)$ for all $b\in M'$. Then, there is a map $\tilde
  f\colon M'' \rightarrow S^4$ that simultaneously is a Darboux transform of
  $f^\sharp$ and $f^\flat$.
\end{theorem}

\begin{proof}
  By Proposition~\ref{prop:global_DT_sections_with_monodromy}, corresponding
  to the Darboux transforms $f^\sharp$ and $f^\flat$ there are holomorphic
  sections $\psi^\sharp$ and $\psi^\flat$ with monodromy of $V/L$ whose
  prolongations $\hat \psi^\sharp$ and $\hat \psi^\flat$ span the line
  subbundles $L^\sharp$ and $L^\flat$ of the trivial $\H^2$--bundle given by
  $f^\sharp$ and $f^\flat$. The idea of the proof is to show that---with the
  right interpretation---the same formula
  \begin{equation}
    \label{eq:bianchi}
    \hat \varphi = \hat \psi^\flat + \hat \psi^\sharp \chi
  \end{equation}
  as in the smooth case defines the prolongation $\hat \varphi$ of a
  holomorphic section $\varphi$ of $V/L^\sharp$. Recall that the prolongations
  $\hat \psi^\flat$ and $\hat \psi^\sharp$ are defined over the vertices of
  $M'$, i.e., over the black triangles of $M$. For \eqref{eq:bianchi} to make
  sense we need $\hat \varphi$ to be a well defined section with monodromy of
  the trivial $\H^2$--bundle over the black triangles of $M'$ alias white
  triangles of $M$.

  The proof is straightforward once we have explained how to make sense of
  \eqref{eq:bianchi}. We do this by proving that there is a unique
  quaternionic function $\chi$ defined on the white triangles of the universal
  covering of $\tilde M$ with the property that 
  \[ \hat \psi^\flat_{1} + \hat \psi^\sharp_{1} \chi_{w} = \hat \psi^\flat_{2}
  + \hat \psi^\sharp_{2} \chi_{w} = \hat \psi^\flat_{3} + \hat \psi^\sharp_{3}
  \chi_{w} \] for every white triangle $w\in \tilde{\mathcal{W}}$, where the
  black triangles adjacent to $w$ are denoted by $1$,...,$3$, see
  Figure~\ref{fig:bianchi}.
  \begin{figure}[hbt]
    \centering \resizebox{3cm}{2.8cm}{\includegraphics{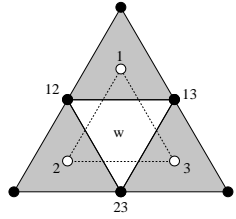}}
    \caption{A white triangle with adjacent vertices and black triangles.}
    \label{fig:bianchi}
  \end{figure}

  Existence and uniqueness of $\chi_w$ follows from the fact that for $i$,
  $j\in \{1,2,3\}$, $i\neq j$
  \begin{equation}
    \hat \psi^\flat_{i} + \hat \psi^\sharp_{i}
    \chi_{ij} = \hat \psi^\flat_{j} + \hat \psi^\sharp_{j} \chi_{ij} \tag{$*$}
  \end{equation}
  is equivalent to
  \[ \hat \psi^\flat_{i} - \hat \psi^\flat_{j} = (\hat \psi^\sharp_{j}-\hat
  \psi^\sharp_{i}) \chi_{ij} \tag{$**$}. \] Because $L^\sharp$ is immersed and
  both sides of $(**)$ are contained in the line $L_{v_{ij}}\subset \H^2$,
  where $v_{ij}$ denotes the vertex of $M$ between the black triangles $i$ and
  $j$ (see Figure~\ref{fig:bianchi}), there is a unique $\chi_{ij}\in \H$ such
  that $(*)$ holds.  To see that $\chi_{ij}$ coincides for all $i$, $j\in
  \{1,2,3\}$, $i\neq j$, note that, since $\hat\psi^\flat$ and $\hat
  \psi^\sharp$ are prolongations of holomorphic sections with monodromy
  of~$V/L$,
  \[ \hat \psi^\flat_{i} + \hat \psi^\sharp_{i} \chi_{12} \equiv \hat
  \psi^\flat_{3} + \hat \psi^\sharp_{3} \chi_{12} \quad \textrm{ mod } \quad
  L_{i3}
  \]
  for $i=1$ and $i=2$.  Because $L_{13}\oplus L_{23}=\H^2$ and $\hat
  \psi^\flat_{1} + \hat \psi^\sharp_{1} \chi_{12}= \hat \psi^\flat_{2} + \hat
  \psi^\sharp_{2} \chi_{12}$ we indeed have $\chi_{i3}=\chi_{12}$ for $i=1$, $2$
  such that $\chi_{ij}$ depends only on the white triangle $w$.  It therefore
  makes sense to defines $\chi_{w}:=\chi_{ij}$ such that \eqref{eq:bianchi}
  becomes
  \[ \hat \psi_w = \hat \psi^\flat_{1} + \hat \psi^\sharp_{1} \chi_{w} = \hat
  \psi^\flat_{2} + \hat \psi^\sharp_{2} \chi_{w} = \hat \psi^\flat_{3} + \hat
  \psi^\sharp_{3} \chi_{w}
  \]
  and in particular yields a well defined section $\hat \varphi$ of the
  trivial $\H^2$--bundle over $M''$ which, by $(**)$, has the same monodromy
  as $\hat \psi^\flat$. By construction, $\hat \varphi$ is the prolongation of
  a holomorphic section $\varphi$ of $V/L^\sharp$ which is nowhere vanishing,
  because $L^\sharp \oplus L^\flat$. In particular, $\hat \varphi$ is nowhere
  vanishing and $\tilde L=\hat \varphi \H$ defines a Darboux transform of
  $f^\sharp$.

  Since $f^\flat$ is immersed, equation $(**)$ implies that the function
  $\chi$ is nowhere vanishing and therefore $\varphi \chi^{-1}= \hat
  \psi^\flat \chi^{-1} + \hat \psi^\sharp$ is also a well defined section of
  the trivial $\H^2$--bundle over~$M''$. It has the same monodromy as $\hat
  \psi^\sharp$ and is the prolongation of a nowhere vanishing holomorphic
  section of $V/L^\flat$.  This shows that $\tilde L$ is also a Darboux
  transform of $f^\flat$.
\end{proof}

Formula \eqref{eq:bianchi} in the proof of the preceding theorem does also
prove the following lemma, because for defining $\chi$ via $(**)$ it is
sufficient that $\hat \psi^\sharp$ corresponds to an immersed Darboux
transform and $\hat \psi^\flat$ is the prolongation of a holomorphic section
with monodromy.

\begin{lemma}\label{lemma:bianchi}
  Let $f\colon M \rightarrow S^4$ be an immersion of a discrete surface $M$
  with regular combinatorics and $f^\sharp\colon M' \rightarrow S^4$ an
  immersed Darboux transform of $f$. For every holomorphic sections
  $\psi^\flat$ with monodromy $h^\flat$ of $V/L$, there is a holomorphic
  section $\varphi$ with the same monodromy $h^\flat$ of $V/L^\sharp$.
\end{lemma}

Together with Corollary~\ref{cor:number_of_infinities} this implies the
invariance of the spectral curve under Darboux transformations.

\begin{theorem}
  The spectral curve $\Sigma$ of an immersion $f\colon M \rightarrow S^4$ of a
  discrete torus $M$ with regular combinatorics that satisfies
  \eqref{eq:genericity_assumption} is preserved under Darboux transforms.
\end{theorem}

\section{Polygons as thin cylinders}\label{sec:curves}
We show now that a discrete curve in the conformal 4--sphere $S^4$, provided
it is a polygon in the sense that the images of every three consecutive points
are mutually disjoint, can be treated as an immersion of a discrete surface of
special type called a thin cylinder. In particular, a closed polygon can be
treated as an immersion of a thin torus, a special type of discrete torus.  We
introduce a Darboux transformation for immersed curves in~$S^4$. In case of
polygons this transformation coincides with the Darboux transformation of
Section~\ref{sec:darboux} applied to the corresponding thin cylinders in
$S^4$.

\subsection{A Darboux transformation for discrete curves in $S^4$}
We generalize the Darboux transformation for curves in $S^2=\CP^1$
\cite{BP96,BP99,HMNP01,Pi02} to curves in $S^4=\HP^1$.  As a special case this
includes, up to translations of $\R^3=\Im \H$ by adding real numbers, the
doubly discrete smoke ring flow introduced by Hoffmann \cite{Ho1,Ho2} for arc
length parametrized curves in $\R^3=\Im \H$ and generalized in \cite{PSW} to
arbitrary curves in $\R^3$.

By \emph{discrete curve} we mean a map $\gamma \colon I\cap \Z \rightarrow
S^4$ defined on the intersection of some interval $I\subset \R$ with $\Z$. In
order to simplify notation, we adopt the convention to drop indices and denote
the points $\gamma_n$ on a curve and their successors $\gamma_{n+1}$,
$\gamma_{n+2}$,... simply by $\gamma$ and $\gamma_+$, $\gamma_{++}$,...  For
example, we call a discrete curve \emph{immersed} if $\gamma\neq \gamma_+$ at
all points and we call it a \emph{polygon} if in addition $\gamma\neq
\gamma_{++}$.

An immersed curve $\eta$ in $S^2=\CP^1$ is a \emph{Darboux transform} of an
immersed curve $\gamma$ if all quadrilaterals spanned by edges of the curve
$\gamma$ and the corresponding edges of the transformed curve $\eta$ have the
same cross--ratio, see Figure~\ref{fig:curve_DT}, that is, if
\begin{figure}[hbt]
  \centering
  \resizebox{7cm}{!}{\input{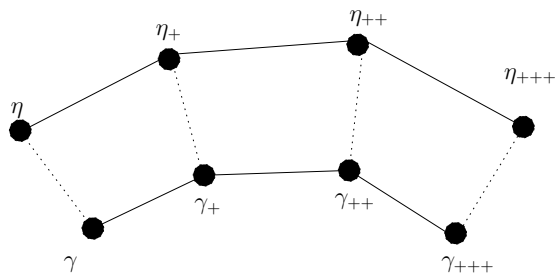}}
  \caption{A Darboux transform of a discrete curve.}
  \label{fig:curve_DT}
\end{figure}
\begin{equation}
  \label{eq:cross_ratio}
  M_4(\gamma,\eta_+,\gamma_+,\eta)=\lambda 
\end{equation}
for some constant $\lambda\in \C_*$, where $M_4$ denotes the cross--ratio
\[M_4(z_1,z_2,z_3,z_4) = (z_1-z_2)(z_2-z_3)^{-1}(z_3-z_4)(z_4-z_1)^{-1}.\]
Note that the cross--ration $M_4(z_1,z_2,z_3,z_4)$ is the image of $z_4$ under
the unique projective transformation $z\mapsto M_4(z_1,z_2,z_3,z)$ that maps
$z_1$, $z_2$ and $z_3$ to the points $\infty$, $1$, $0$.

The Darboux transform $\eta$ is uniquely determined by the cross ratio
$\lambda\in \C_*$ together with an initial value, say $\eta_0$: the other
points of $\eta$ are determined by the recursion formula
\begin{equation}
  \label{eq:recursion_curve_dt}
  \eta_+ = (P+\lambda\, Q) \eta,
\end{equation}
where $P$ and $Q$ denote the projections from $\C^2$ to the summands of the
splitting given by $\gamma$ and $\gamma_+$. To check this, assume that
$\gamma$, $\gamma_+$ and $\eta$ have homogeneous coordinates $[1,0]$, $[0,1]$
and $[\lambda,1]$. Then
\[ P+\lambda Q =
\begin{pmatrix}
  1 & 0 \\ 0 & \lambda 
\end{pmatrix} \] and the point $\eta_+$ given by \eqref{eq:recursion_curve_dt}
has homogeneous coordinates $[\lambda,\lambda]=[1,1]$ such that the
cross--ration condition \eqref{eq:cross_ratio} is indeed satisfied.

For generalizing this cross--ration evolution to curves in $S^4=\HP^1$ one has
to deal with the problem that prescribing 3 points in $S^4$ and a non--real
cross--ratio does only determine a unique forth point if one additionally
prescribes an oriented 2--sphere. This problem can be overcome by adapting the
recursion formula \eqref{eq:recursion_curve_dt} to the quaternionic setting:
in order to allow the use of complex spectral parameters $\lambda\in \C_*$ we
view $\H^2$ as the 4--dimensional complex vector space $(\H^2,\i)$ with right
multiplication by complex numbers $\C=\Span_\R\{1,\i\}\subset \H$.  Then
\eqref{eq:recursion_curve_dt} can be seen as a recursion formula for
homogeneous lifts $\hat \eta$ of the Darboux transform $\eta$:
\begin{definition}
  A curve $\eta$ in $S^4=\HP^1$ is a \emph{Darboux transform} of a curve
  $\gamma$ in $S^4=\HP^1$ if there is $\lambda\in \C_*$ and a homogeneous lift
  $\hat \eta$ of $\eta$ to $\H^2$ that satisfies
\begin{equation}
  \label{eq:recursion_curve_dt_quat}
  \hat \eta_+ = (P+\lambda\, Q) \hat \eta,
\end{equation}  
where $P$ and $Q$ denote the (quaternionic linear) projections operators from
$\H^2$ to the summands of the splitting given by $\gamma$ and $\gamma_+$ and
where $\lambda$ stands for the complex (but not quaternionic) linear operator
obtained by multiplication with $\lambda\in \C_*$ on the space $(\H^2,\i)$, in
other words $\hat \eta_+=P\hat \eta + Q\hat \eta \lambda$.
\end{definition}

Because \eqref{eq:recursion_curve_dt_quat} is not a quaternionic linear
equation, the initial value that determines a Darboux transform is not the
point $\eta=\hat \eta \H\in \HP^1$, but a ``twistor lift'' $\hat
\eta\C\in\CP^3$.  This twistor lift $\hat \eta\in\CP^3$ has the following
geometric interpretation: there is a unique oriented 2--sphere through
$\gamma$, $\gamma_+$ and $\eta$ obtained as the image of the complex line
$\hat \gamma\C\oplus\hat \gamma_+\C$ under the twistor projection $\pi\colon
\CP^3\rightarrow \HP^1$, where $\hat \gamma$ and $\hat \gamma_+$ are
homogeneous lifts of $\gamma$ and $\gamma_+$ with $\hat \eta = \hat \gamma +
\hat \gamma_+ $. Obviously, the point $\eta_+$ obtained by
\eqref{eq:recursion_curve_dt_quat} is also contained in this 2--sphere such
that the cross--ration of the four points is a well defined complex number
and, by the above argument for the complex case, indeed $\eta_+$ is the unique
forth point on that oriented 2--sphere with
$M_4(\gamma,\eta_+,\gamma_+,\eta)=\lambda$.  The recursion formula
\eqref{eq:recursion_curve_dt_quat} defines not only $\eta_+$, but a twistor
lift $\hat\eta_+$ and hence a new oriented 2--sphere through $\gamma_+$,
$\gamma_{++}$ and $\eta_+$, a new point $\eta_{++}$ on that 2--sphere and so
forth... In addition to the Darboux transform $\eta$ our construction yields a
congruence of oriented 2--spheres along the edges of the curves.  The
intersection of two consecutive 2--spheres describes the initial curve
$\gamma$ and the Darboux transform $\eta$.

In case that $\gamma\colon \Z\rightarrow S^4$ is a \emph{closed} curve with
period $n$ one is mainly interested in closed Darboux transformations. These
are obtained by taking as initial conditions for Darboux transforms with
parameter $\lambda\in \C_*$ the eigenlines of the holonomy
\begin{equation}
  \label{eq:holonomy_curves}
  H(\lambda) = (P_{n-1}+\lambda\, Q_{n-1})\cdot(P_{n-2}+\lambda\,
  Q_{n-2})\cdot...\cdot(P_1+\lambda \, Q_1) \cdot(P_0 +\lambda\,  Q_0).
\end{equation}

A natural thing to do when coming across such polynomial family $\lambda
\mapsto H(\lambda)$ of endomorphisms is to study its eigenspaces and their
dependence on $\lambda$. The vanishing locus $\det(\mu-H(\lambda))=0$ of the
characteristic polynomial is an algebraic curve and the normalization of
$\{(\mu,\lambda)\in \C_* \times \C_* \mid \det(\mu-H(\lambda))=0\}$ is the
\emph{spectral curve} $\Sigma$ of the closed immersed discrete curve $\gamma$.
Proposition~\ref{prop:fredholm} implies that the eigenspaces of $H(\lambda)$
extend to  holomorphic vector bundles over the components of the Riemann
surface $\Sigma$.

For a generic closed curve in $S^4$ on expects the holonomies $H(\lambda)$ to
have simple eigenvalues for all but finitely many parameters $\lambda$.  The
eigenspaces then extend to an ``eigenline bundle'' over $\Sigma$ and the
spectral curve parametrizes the closed Darboux transforms of $\gamma$. In
particular, the spectral curve $\Sigma$ is then a 4--fold branched covering of
the $\lambda$--plane.

As in the case of discrete tori in $S^4$, the spectral curve of a closed curve
in $S^4$ has an anti--holomorphic involution $\rho\colon \Sigma \rightarrow
\Sigma$. It covers $\lambda \mapsto \bar \lambda$ and is induced by $H(\bar
\lambda)= j^{-1}H(\lambda)j$, where $j$ denotes the complex anti--linear
operator given by right multiplication with the quaternion $j$. 

\begin{theorem} \label{th:curve_spec} Let $\gamma$ be a closed polygon in
  $S^4$ for which the holonomy $H(\lambda)$ generically has four simple
  eigenvalues. The spectral curve of $\gamma$ has three pairs $\infty_+$,
  $\rho\infty_+$, $\infty_0$, $\rho\infty_0$, $\infty_-$, and $\rho \infty_-$
  of points at infinity such that, for each point of the curve, the twistor
  projection of the eigenline curve evaluated at the respective point at
  infinity is $\gamma_+$, $\gamma$ and~$\gamma_-$.
\end{theorem}
\begin{proof}
  For investigating the asymptotics of the spectral curve $\Sigma$ it is
  sufficient to study the coefficients with lowest and highest order of the
  polynomial $H(\lambda)$.  The lowest order coefficient of $H(\lambda)$ is
  \[ H(0) = P_{n-1}\cdot P_{n-2}\cdot\cdot\cdot P_1\cdot P_0. \] The subspaces
  $L_{1}$ and $L_{n-1}$ are invariant under $H(0)$ which vanishes on $L_1$ and
  acts non--trivial on $L_{n-1}$, because $\gamma$ is assumed to by a polygon.

  For curves of even length, the highest order coefficient is
  \begin{align*}
    H_{max} & = Q_{n-1}\cdot P_{n-2}\cdot\cdot\cdot P_2\cdot Q_1\cdot P_0 +
    P_{n-1}\cdot Q_{n-2}\cdot\cdot\cdot Q_2 \cdot P_1\cdot Q_0
    = \\
    & = Q_{n-1}\cdot Q_{n-3}\cdot\cdot\cdot Q_3\cdot Q_1\cdot P_0 +
    Q_{n-2}\cdot Q_{n-4} \cdot\cdot\cdot Q_2\cdot Q_0,
  \end{align*}
  because $Q_{l+1}\cdot Q_l=0$ and $P_{l+1}\cdot Q_{l}=Q_l$.  The endomorphism
  $H_{max}$ leaves $L_0$ invariant and maps $L_1$ to $L_{n-1}$.  Because
  $\gamma$ is a polygon, its restriction to both summands $L_0$ and $L_1$ is
  non--trivial and $H_{max}$ is invertible.

  For curves of odd length, the highest order term is
  \[ H_{max} = Q_{n-1}\cdot Q_{n-2} \cdot \cdot\cdot Q_2 \cdot Q_0. \] It
  vanishes on $L_0$ and its restriction to $L_1$ is non--trivial (because
  $\gamma$ is a polygon) and maps $L_1$ to $L_0$. In particular, $H_{max}$ is
  nil--potent.
\end{proof}

For curves of odd length, the last part of the proof suggests that the
spectral curve has two branch points, $\infty_0$ and $\rho\infty_0$, over
$\lambda=\infty$ such that there are no other points at infinity apart from
those described in Theorem~\ref{th:curve_spec}.

\begin{rem}\label{rem:spec_s2}
  Analogous to Lemma~\ref{lem:two--sphere}, the spectral curve of a closed,
  immersed discrete curve in $S^2=\CP^1$, when seen as a curve in $S^4$,
  consists of two copies of the $\CP^1$--spectral curve which are interchanged
  under $\rho$.  For polygons in $S^2=\CP^1$, as in the proof of
  Theorem~\ref{th:curve_spec} the asymptotics at $\lambda=0$ shows that the
  holonomy, which in this case reduces to a family of endomorphisms of a
  complex rank two vector space, generically has simple eigenvalues. In
  particular, a closed polygon $\gamma$ in $\CP^1$ always admits an eigenline
  bundle over its spectral curve which parametrizes the eigenlines of
  $H(\lambda)$ and hence closed Darboux transforms of $\gamma$.
\end{rem}

\subsection{Thin cylinders}
In order to interpret the parameter domain of a discrete curve as a \emph{thin
  cylinder}, a special type of discrete surface, we take one row of black and
white triangles in the regular triangulation of the plane and identify the
lower left point of each black triangle with the upper right point of the
white triangle right of it. Figure~\ref{fig:thin_torus} shows as an example
the fundamental domain of a \emph{thin torus} that is parameter domain for
closed curves of period $n$.

\begin{figure}[hbt]
  \centering
  \resizebox{9cm}{!}{\includegraphics{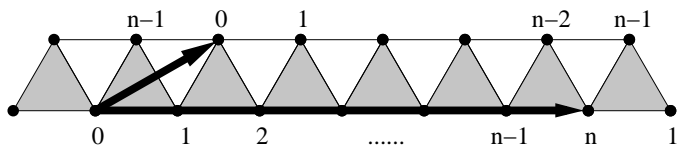}}
  \caption{A thin torus with $n$ points.}
  \label{fig:thin_torus}
\end{figure}

An immersion of a thin cylinder is an immersed discrete curve $\gamma$ which
is a polygon in the sense that every three consecutive points of $\gamma$ are
mutually different, that is, in addition to $\gamma\neq \gamma_+$ the curve
satisfies $\gamma\neq \gamma_{++}$ for all points.

If $M$ is a thin cylinder, then $M'$ and $M''$ introduced in
Section~\ref{sec:three_gen} are also thin cylinders.  The numbering of the
vertices as in Figure~\ref{fig:thin_torus} induces a numbering of the black
and white triangles (and hence of the vertices of $M'$ and $M''$): a black
triangle gets the same number than its lower left vertex while a white
triangle gets the same number than its upper right vertex.  With this
numbering convention, all maps from $M$, $M'$ and $M''$ to $S^4$ are curves
defined on the same parameter domain. We prove now that a map $f^\sharp\colon
M'\rightarrow S^4$ is the Darboux transform of an immersion $f\colon
M\rightarrow S^4$ of a thin cylinder $M$ if the curve corresponding to
$f^\sharp$ is a Darboux transform of the polygon corresponding to~$f$:

\begin{theorem}\label{th:curve_DT_as_cylinder_DT} 
  Let $\gamma\colon I\cap \Z\rightarrow S^4$ be a polygon in $S^4$. A discrete
  curve $\eta\colon I\cap \Z\rightarrow S^4$ is a Darboux transform of
  $\gamma$ if and only if the thin cylinder in~$S^4$ corresponding to $\eta$
  is a Darboux transform of the immersed thin cylinder corresponding to
  $\gamma$.  In particular, the spectral curve of a closed polygon $\gamma$ in
  $S^4$ coincides with the spectral curve of the corresponding thin torus in
  $S^4$.
\end{theorem}

\begin{proof}
  Let $\hat \psi$ be the prolongation of a holomorphic section with complex
  monodromy on the universal covering of a thin cylinder and let $\lambda\in
  \C_*$ be its ``vertical'' monodromy (i.e., the monodromy in direction of the
  lattice vector that identifies the lower left vertex of a black triangle
  with the upper right vertex of the white triangle right of it). As in
  Figure~\ref{fig:thin_cylinder}, we denote by $\hat \psi$, $\hat\psi_+$,...
  the section $\hat \psi$ on the black triangles with lower left corner
  $\gamma$, $\gamma_+$,... along one row of black triangles.

\begin{figure}[hbt]
  \centering
  \resizebox{3.5cm}{!}{\input{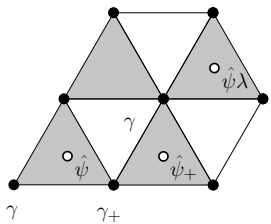}}
  \caption{Darboux transformation of a piece of a thin cylinder.}
  \label{fig:thin_cylinder}
\end{figure}

Because the projections of $\hat \psi_+$ and $\hat \psi \lambda$ to
$(V/L)_{|\gamma}$ coincide, we have
\[ \hat \psi_+ \equiv \hat \psi \lambda \mod L_{|\gamma} \tag{$*$}, \] and,
because the projections of $\hat \psi_+$ and $\hat \psi$ to
$(V/L)_{|\gamma_+}$ coincide, we have
\[ \hat \psi_+ \equiv \hat \psi \mod L_{|\gamma_+}. \tag{$**$} \] Together,
equations $(*)$ and $(**)$ are equivalent to
\[ \hat \psi_+ = P\hat \psi + Q\hat \psi \lambda\] with $P$ and $Q$ denoting
the quaternionic linear projections from $V$ to the summands of the splitting
induced by $\gamma$ and $\gamma_+$.

\end{proof}



\begin{thebibliography}{99}

\bibitem{BHS02} A.~I.~Bobenko, T.~Hoffmann, and Yu.~B.~Suris, Hexagonal circle
  patterns and integrable systems: {P}atterns with the multi-ratio property
  and Lax equations on the regular triangular lattice. \emph{IMRN} \textbf{3}
  (2002), 111--164.

\bibitem{BMS05} A.~I.~Bobenko, C.~Mercat, and Yu.~B.~Suris, Linear and
  nonlinear theories of discrete analytic functions. Integrable structure and
  isomonodromic Green's function.  \emph{J.~reine angew.~Math.} \textbf{583}
  (2005), 117--161.

\bibitem{BP96} A.~I.~Bobenko and U.~Pinkall, Discrete isothermic
  surfaces. \emph{J.~reine angew.~Math.} \textbf{478} (1996), 187--208.

\bibitem{BP99} A.~I.~Bobenko and U.~Pinkall, Discretization of surfaces and
integrable systems. In A.~I.~Bobenko and R.~Seiler, \emph{Discrete Integrable
Geometry and Physics}, Clarendon Press, Oxford, 1999.

\bibitem{BHS06} A.~I.~Bobenko, T.~Hoffmann, and B.~A.~Springborn, Minimal
  surfaces from circle patterns: geometry from combinatorics.  \emph{Ann. of
    Math.}  \textbf{164} (2006), 231--264.

\bibitem{Boh03} C.~Bohle, \emph{M{\"o}bius Invariant Flows of Tori in $S^4$}.
  PhD thesis, Technische Universit\"at Berlin, 2003.

\bibitem{BLPP} C.~Bohle, K.~Leschke, F.~Pedit, and U.~Pinkall, Conformal maps
  from a 2--torus to the 4--sphere.  
  \url{http://xxx.arxiv.org/abs/0712.2311}.

\bibitem{BPP} C.~Bohle, F.~Pedit, and U.~Pinkall, Spectral curves of
  quaternionic holomorphic line bundles over 2--tori.  Preprint.

\bibitem{BP} C.~Bohle and U.~Pinkall, Discrete holomorphic geometry II.
  Weierstrass representation and minimal surfaces.  In Preparation.


\bibitem{BPP01} F.~E.~Burstall, F.~Pedit, and U.~Pinkall, Schwarzian
  derivatives and flows of surfaces. In  \emph{Differential Geometry and
    Integrable Systems} edited by M.A.~Guest, R.~Miyaoka, Y.~Ohnita,
  \emph{Contemp.~Math.} \textbf{308} (2002), 39--61.

\bibitem{BFLPP02}
F.~E.~Burstall, D.~Ferus, K.~Leschke, F.~Pedit, and U.~Pinkall,
\emph{Conformal Geometry of Surfaces in $S^4$ and Quaternions}.
Lecture Notes in Mathematics 1772, Springer, Berlin, 2002.

\bibitem{CdV98} Y.~Colin de Verdi\`ere, Multiplicities of eigenvalues and
  tree--width of graphs.  \emph{J.~Combin.~Theory Ser.~B} \textbf{74} (1998),
  121--146.

\bibitem{DS} A.~Davey and K.~Stewartson, On three--dimensional packets of
  surface waves. \emph{Proc.~Roy.~Soc.~London~A} \textbf{338} (1974),
  101--110.
 
\bibitem{DN97} I.~A.~Dynnikov and S.~P.~Novikov, Discrete spectral
symmetries of low--dimensional differential operators and difference
operators on regular lattices and two--manifolds. \emph{Russian Math.\
Surveys} \textbf{52} (1997), 1057--1116.

\bibitem{DN03} 
I.~A.~Dynnikov and S.~P.~Novikov, 
Geometry of the triangle equation on two--manifolds.  
\emph{Mosc.~Math.~J.} \textbf{3} (2003), 419--438.

\bibitem{FLPP01}
D.~Ferus, K.~Leschke, F.~Pedit, and U.~Pinkall,
Quaternionic holomorphic geometry: Pl\"ucker formula, Dirac eigenvalue
estimates and energy estimates of harmonic 2--tori.
\emph{Invent.\ Math.} \textbf{146} (2001), 507--593.

\bibitem{GH} P.~Griffiths and J.~Harris, \emph{Principles of Algebraic
    Geometry.} Wiley, New York, 1978.

\bibitem{GS98}
P.~G.~Grinevich and M.~U.~Schmidt,
Conformal invariant functionals of immersions of tori into $R\sp 3$.  
\emph{J.\ Geom.\ Phys.} \textbf{26} (1998), 51--78.

\bibitem{HMNP01} U.~{Hertrich-Jeromin}, I.~{McIntosh}, P.~Norman, and
F.~Pedit, Periodic discrete conformal maps. \emph{J.~reine angew.~Math.}
\textbf{534} (2001), 129--153.

\bibitem{Ho1} T.~Hoffmann. \emph{Discrete curves and surfaces.} PhD thesis,
  Technische Universit\"at Berlin, 2000.

\bibitem{Ho2} T.~Hoffmann, Discrete Hashimoto surfaces and a doubly discrete
  smoke ring flow. To appear in Bobenko et al., \emph{Lectures on Discrete
    Differential Geometry}, Oberwolfach Seminars, Birkh\"auser,
  Basel. \url{http://xxx.arxiv.org/abs/math/0007150}.

\bibitem{Ke02} R.~Kenyon, The Laplacian and Dirac operators on critical planar
  graphs. \emph{Invent.~Math.}  \emph{150} (2002), 409--439.

\bibitem{Ko96}
B.~G.~Konopelchenko,
Induced surfaces and their integrable dynamics.
\emph{Stud.\ Appl.\ Math.} \textbf{96} (1996), 9--51.

\bibitem{Ko00}
B.~G.~Konopelchenko, 
Weierstrass representations for surfaces in 4D spaces and their integrable
deformations via DS hierarchy.
\emph{Ann.\ Glob.\ Anal.\ Geom.} \textbf{18} (2000), 61--74.

\bibitem{KS01} B.~G.~Konopelchenko and W.~K.~Schief, Menelaus' theorem,
Clifford configurations and inversive geometry of the Schwarzian KP hierarchy.
\emph{J.~Phys.~A} \textbf{35} (2001), 6125--6144.


\bibitem{KS03} B.~G.~Konopelchenko and W.~K.~Schief, Conformal geometry of the
  (discrete) {Schwarzian Davey--Stewartson} {II} hierarchy. \emph{Glasgow
    Math.~J.} \textbf{47} (2005) 121-131.

\bibitem{Kr89} I.~Krichever, Spectral theory of two--dimensional periodic
  operators and its applications. \emph{Russian Math. Surveys} \textbf{44}
  (1989), 145--225.

\bibitem{Ku93} P.~Kuchment, \emph{Floquet theory for partial differential
    equations.}  Birkh\"auser, Basel, 1993.

\bibitem{Me} C.~Mercat, Discrete {R}iemann surfaces and the {I}sing model.
  \emph{Comm.~Math.~Phys.} \textbf{218} (2001), 177--216. 

\bibitem{PP98}
F.~Pedit and U.~Pinkall,
Quaternionic analysis on Riemann surfaces and differential geometry.
\emph{Doc.\ Math.\ J.\ DMV}, Extra Volume ICM Berlin 1998, Vol.\ II, 389--400.

\bibitem{Pi02} U.~Pinkall, \emph{Experimentelle Differentialgeometrie}.
  Lecture notes (taken by Torsten Senkbeil) on the integrable geometry of
  discrete curves in $\CP^1$, with Java--Applets, TU--Berlin, 2002.

\bibitem{PSW} U.~Pinkall, B.~Springborn and S.~Wei\ss mann, A new doubly
  discrete analogue of smoke ring flow and the real time simulation of fluid
  flow. To appear in \emph{J.\ Phys.\ A}.

\bibitem{S02} M.~U.~Schmidt, A proof of the Willmore conjecture.
  \url{www.arxiv.org/abs/math/0203224}.

\bibitem{St} K.~Stephenson, \emph{Introduction to circle packing. The theory
    of discrete analytic functions.}  Cambridge University Press, 2005.

\bibitem{Ta97}
I.~A.~Taimanov, 
Modified Novikov--Veselov equation and differential geometry of surfaces. 
\emph{Amer.\ Math.\ Soc.\ Transl.} \textbf{179} (1997), 133--151.

\bibitem{Ta98}
I.~A.~Taimanov,
The Weierstrass representation of closed surfaces in $R\sp 3$.
\emph{Funct.\ Anal.\ Appl.}  \textbf{32} (1998), 258--267.


\end{thebibliography}
\end{document}